\documentclass[10pt]{article}
\usepackage{verbatim}

\usepackage{amsmath}
\usepackage{amssymb}
\usepackage{latexsym}
\usepackage{staves}
\usepackage{marvosym}
\usepackage{eurosym}
\usepackage{phaistos}
\usepackage{verbatim}
\usepackage{ulem}
\usepackage{bm}
\usepackage{amsthm}
\usepackage{color}
\usepackage[T1]{fontenc}
\usepackage[english]{babel}
\usepackage{graphicx}
\usepackage{tikz}
\usetikzlibrary{arrows}
\usetikzlibrary{arrows,decorations.markings}
\usetikzlibrary{mindmap,trees}

\newtheorem{Theorem}{Theorem}[part]
\newtheorem{Definition}{Definition}[part]
\newtheorem{Proposition}{Proposition}[part]

\newtheorem{Assumption}{Assumption}[part]
\newtheorem{Lemma}{Lemma}[part]

\newtheorem{Remark}{Remark}[part]

\makeatletter \@addtoreset{equation}{section}

\@addtoreset{Definition}{section}

\@addtoreset{Theorem}{section}

\@addtoreset{Proposition}{section}

\@addtoreset{Property}{section}

\@addtoreset{Assumption}{section}

\@addtoreset{Corollary}{section}

\@addtoreset{Lemma}{section}

\@addtoreset{Remark}{section}

\@addtoreset{Example}{section}

\def\esssup{{\rm ess}\,\sup\limits}

\def\Fc{{\cal F}}

\def\05{\frac{1}{2}}
\def\-1{^{-1}}
\def\1{{1\hspace{-1mm}{\rm I}}}
\def\={\;=\;}
\def\.{\;.}

\title{  }
\author{ }

\def\be{\begin{eqnarray}}
\def\ee{\end{eqnarray}}

\def\b*{\begin{eqnarray*}}
\def\e*{\end{eqnarray*}}

\def\E{\mathbb{E}}

\def\R{\R}

\def\P{\mathbb{P}}

\def \ind{{\bf 1}}

\def \R{I\!\!R}

\def\Fc{{\cal F}}

\def\-1{^{-1}}

\def\0.5{\frac{1}{2}}

\def\={\;=\;}
\def\.{\;.}

\def\1{{\bf 1}}

\def\Ybf{{\bf Y}}

\def \sp {\mathcal{S}^p}
\def \spn {\mathcal{S}^{p,\otimes n}}
\def \lpn {\mathbb{L}^{p,\otimes n}}

\title{A propagation of chaos result for weakly interacting nonlinear Snell envelopes
}
\author{Boualem Djehiche \thanks{Department of Mathematics, KTH Royal Institute of Technology, Stockholm, Sweden, email: \texttt{boualem@kth.se}} \and Roxana Dumitrescu \thanks{Department of Mathematics, King's College London, United Kingdom, email: \texttt{roxana.dumitrescu@kcl.ac.uk}} \and Jia Zeng\thanks{Department of Mathematics, King's College London, United Kingdom and 
The University of Hong Kong, Hong Kong, email: \texttt{jia.zeng@kcl.ac.uk}}}

\begin{document}

\maketitle

\begin{abstract}
    In this article, we establish a propagation of chaos result for weakly interacting nonlinear Snell envelopes which converge to a class of mean-field reflected backward stochastic differential equations (BSDEs) with jumps and right-continuous and left-limited obstacle, where the mean-field interaction in terms of the distribution of the $Y$-component of the solution enters both the driver and the lower obstacle. Under mild Lipschitz and integrability conditions on the coefficients, we prove existence and uniqueness of the solution to both the mean-field reflected BSDEs with jumps and the corresponding system of weakly interacting particles and  provide a propagation of chaos result for the whole solution $(Y,Z,U,K)$, which requires new technical results due to the dependence of the obstacle on the solution and the presence of jumps.
\end{abstract}



\section{Introduction}

The notion of propagation of chaos can be traced back at least to the work by Kac \cite{k56} whose initial motivation was to investigate the particle system approximation of some nonlocal partial differential equations (PDEs) arising in thermodynamics. He studied the evolution of a large number $n$ of random particles starting from $n$ given independent and identically distributed random variables, i.e. the initial configuration is chaotic, and whose respective dynamics weakly interact. Kac's insight is that the initial chaotic configuration propagates over time. Precise mathematical treatments of this theory can be found in McKean \cite{m67}, Sznitman \cite{s91} and G{\"a}rtner \cite{g88}. For further development and applications of propagation of chaos theory, see e.g. Jabin \cite{j16}, Lacker \cite{l18}, Shkolnikov \cite{s12} and references therein.  More recently, due to the very active development within mean-field type control and games (see Lasry and Lions \cite{ll07}, Andersson and Djehiche \cite{AD11}), there is a renewed interest to study propagation of chaos for stochastic systems with a view towards optimal control and games.

\noindent A question is to ask whether the propagation of chaos carries over to systems of particles with chaotic terminal configuration \textcolor{black}{such as systems of weakly interacting backward stochastic differential equations (BSDEs)}. 

BSDEs have been widely studied since the pioneering work by Pardoux and Peng \cite{pp90}. In particular, they appear as a powerful tool to solve stochastic optimization problems in many fields such as engineering, investment science including mathematical finance, game theory and insurance, among many other areas. In the Markovian case, BSDEs provide a probabilistic representation of a class of semilinear PDEs. Later Tang and Li \cite{tl94} considered the more involved case when the noise is driven not only by a Brownian motion but also by an independent Poisson random measure. They proved existence and uniqueness of the solution. Barles et al. \cite{bbp97} studied the link of those BSDEs with viscosity solutions of integral-partial differential equations. The notion of reflected BSDEs was first introduced in El Karoui et al. \cite{ekppq97}. In the setting of those BSDEs, one of the components of the solution is forced to stay above a given process. The extension to the case of reflected BSDEs with jumps can be found in e.g.  \cite{dl1}, \cite{dl2}, \cite{dqs15}, \cite{dqs16}, \cite{essaky08}, \cite{ho03}, \cite{ho16}.

Mean-field BSDEs have been introduced in Buckdahn et al. \cite{blp09} and motivated by their connection to control of McKean-Vlasov equation or mean-field games, see Carmona et al. \cite{cdl13} and Acciaio et al. \cite{abc19} for the discussions of their connections.  Later, the addition of a reflection to these BSDEs has been considered in Li \cite{l14} and Djehiche et al. \cite{deh19}. In particular, motivated by the setting of dynamical solvency level in life insurance and pension, Djehiche et al. \cite{deh19} studied a new class of mean-field type reflected BSDEs, where the mean-field interaction in terms of the distribution of the $Y$-component of the solution enters both the driver and the lower obstacle. They obtained the well-posedness of such a class of equations under mild Lipschitz, integrability conditions on the coefficients, and a monotonicity assumption on the driver and the obstacle, using the classical penalization scheme.

Limit theorems, including backward propagation of chaos, for weakly interacting particles whose dynamics is given by a system of backward SDEs are obtained in Buckdahn et al. \cite{bdlp09}, Hu {\it et al.} \cite{hry19}, Lauri{\`e}re and Tangpi \cite{lt19} and Briand {\it et al.} \cite{bcdph20} in the case of non-reflected BSDE driven by Brownian motion. Li \cite{l14} extends the results of \cite{bdlp09} to reflected BSDEs where the weak interaction enters only the driver, while the work by Briand and Hibon \cite{bh21} considers a particular class of reflected BSDEs (under rather special conditions of the involved coefficients) where the constraint acts on the law of the $Y$-component rather than on its path for which they show that the associated increasing process is necessarily deterministic.

In this paper, we establish a propagation of chaos result for a class of weakly interacting reflected BSDEs with jumps and right-continuous left-limited obstacles, having the particularity that even the obstacle process depends on both the solution and the mean field interaction term. This extension of the Brownian motion driven models is not only due to a pure mathematical curiosity, but is rather motivated by the recent modeling problems related to systemic risk which effect the critical dynamical solvency levels of life insurance and pension systems (see \cite{deh19} and the references therein). The discontinuity of the recursive utilities or the prospective reserves (solutions of the BSDE), due to both the discontinuous character of claims (which typical occur following a Poisson jump process) and the discontinuity of the obstacle (due to changes of the regulation and economic shocks), cannot be smoothed away by only considering Brownian motion driven models. 

\paragraph{Main contributions.} The present paper brings novel results to the literature on mean-field reflected BSDEs and propagation of chaos type results in several directions. First, we show the wellposedness of a new class of mean-field reflected BSDEs with jumps and right-continuous left-limited obstacle depending on the component $Y$ of the solution and its law by developing a new technique compared to  Djehiche et al. \cite{deh19} in the Brownian case. Our approach allows to get rid of the monotonicity assumptions on the driver and obstacle with respect to the measure component, which represents a major improvement with respect to  Djehiche et al. \cite{deh19}. More precisely, in order to deal with the possible dependence of the obstacle on the solution, we use a new method based on the characterization of the solution of the mean-field reflected BSDE as a nonlinear optimal stopping problem of mean-field type, as well as on appropriate estimates. Furthermore, we establish a sufficient condition on the type of dependence of the obstacle on the solution which ensures the continuity of the increasing predictable process, which pushes the solution above the obstacle in a minimal way. Indeed, the setting with jumps becomes more challenging in this context due to the dependence of the obstacle on the solution and the two possible sources of jumps, at predictable stopping times and totally inaccessible stopping times.

Second, we provide a particle system approximation of the mean-field reflected BSDE with jumps in terms of a system of weakly interacting Snell envelopes and show its well-posedness in a general setting. The interaction of particles is through the empirical distribution of the system. The third main contribution consists in providing some convergence results of the $n$-particle system to the solution of the mean-field reflected BSDE under mild Lipschitz and integrability conditions on the coefficients, as well as related propagation of chaos results. In order to show the convergence of the weakly interacting system of the Snell envelops, we provide a law of large numbers type result for independent copies of the $Y$-component of the solution to the mean-field reflected BSDE with jumps, which seems new in the literature. The presence of jumps brings technical difficulties that we overcome in our proof thanks to a careful study of the Skorohod space of right continuous functions with left limits. To the best of our knowledge, this is the first paper which studies the well-posedness and the propagation of chaos property of the mean-field reflected BSDEs with jumps in a general context.

\paragraph{Organization of the paper.} In Section 2, we prove the existence and uniqueness of the mean-field reflected BSDEs with jumps, where the mean-field interaction in terms of the distribution of the $Y$-component of the solution enters both the driver and the lower obstacle. The well-posedness of the mean-field reflected BSDE with jumps is obtained by extending the notion of Peng's nonlinear expectation (e.g. Peng \cite{Peng04}). In Section 3, we discuss the interpretation of the mean-field reflected BSDEs with jumps at the particle level and show the existence and uniqueness of the system of weakly interacting particles. In Section 4, we give a law of large numbers result for the empirical measure associated with  $\mathbb{D}:=D([0,T],\R)$-valued independent copies of the $Y$-component of the solution to the mean-field reflected BSDE with jumps w.r.t. the $p$-Wasserstein distance. With this law of large numbers, we prove the convergence of the particle system to solutions of mean-field reflected BSDEs with jumps, and obtain a related propagation of chaos result for the solution $(Y,Z,U,K)$.

\paragraph{Notation.}
\quad Let $(\Omega, \mathcal{F},\mathbb{P})$ be a complete probability space.  $B=(B_t)_{0\leq t\leq T}$  is a standard one dimensional Brownian motion and $N(dt,de)$ is a Poisson random measure, independent of $B$, with compensator $\nu(de)dt$ such that $\nu$ is a $\sigma$-finite measure on $\R^*$ equipped with its Borel $\sigma$-field $\mathcal{B}(\R^*)$. Let $\tilde{N}(dt,du)$ be its compensated process. We denote by $\mathbb{F} = \{\mathcal{F}_t\}$ the (completed) natural filtration associated with $B$ and $N$. Let $\mathcal{P}$ be the $\sigma$-algebra on $\Omega \times [0,T]$ of $\mathcal{F}_t$-progressively measurable sets. Next, we introduce the following spaces, for $p\ge 1$ and $\beta\ge 0$:
\begin{itemize}
    \item $\mathcal{T}_t$ is the set of $\mathbb{F}$-stopping times $\tau$ such that
    $\tau \in [t,T]$ a.s.
    \item \textcolor{black}{$L^p(\mathcal{F}_T)$ is the set of random variables $\xi$ which are $\mathcal{F}_T$-measurable and $\mathbb{E}[|\xi|^p]<\infty$.}
    \item $\mathcal{S}_{\beta}^p$ is the set of real-valued c{\`ad}l{\`a}g adapted processes $y$ such that $||y||^p_{\mathcal{S}_{\beta}^p} :=\mathbb{E}[\underset{0\leq u\leq T}{\sup} e^{\beta p s}|y_u|^p]<\infty$. We set $\mathcal{S}^p=\mathcal{S}_{0}^p$.
    
    \item $\mathcal{S}_{\beta,i}^p$ is the subset of $\mathcal{S}_{\beta}^p$ such that the process $k$ is non-decreasing and $k_0 = 0$. We set $\mathcal{S}_{i}^p=\mathcal{S}_{0,i}^p$.
    
    \item $\mathbb{L}_{\beta}^p$ is the set of real-valued c{\`ad}l{\`a}g adapted processes $y$ such that $||y||^p_{\mathbb{L}_{\beta}^p} :=\underset{\tau\in\mathcal{T}_0}{\sup}\E[e^{\beta p \tau}|y_{\tau}|^p]<\infty$. $\mathbb{L}_{\beta}^p$ is a Banach space (see Theorem 22 in \cite{DM82}, pp. 60 when $p=1$). We set $\mathbb{L}^p=\mathbb{L}_{0}^p$. 

    \item $\mathcal{H}^{p,d}$ is the set of $\mathcal{P}$-measurable, $\R^d$-valued processes such that  $\mathbb{E}[(\int_0^T|v_s|^2ds)^{p/2}] <\infty$.
    
     \item $L^p_{\nu}$ is the set of measurable functions $l:\R^*\to \R$ such that $\int_{\R^*}|l(u)|^p\nu(du)<+\infty$.
    The set $L^2_{\nu}$ is a Hilbert space equipped with the scalar product $\langle\delta,l\rangle_{\nu}:=\int_{\R^*}\delta(u)l(u)\nu(du)$ for all $(\delta,l)\in L^2_{\nu}\times L^2_{\nu}$, and the norm $|l|_{\nu,2}:=\left(\int_{R^*}|l(u)|^2\nu(du)\right)^{1/2}$. If there is no risk for confusion, we sometimes denote $|l|_{\nu,2}:=|l|_{\nu}$.
    
    \item $\mathcal{B}(\R^d)$ (resp. $\mathcal{B}(L^p_{\nu}))$ is the Borel $\sigma$-algebra on $\R^d$ (resp. on $L^p_{\nu}$).
    
    \item $\mathcal{H}^{p,d}_{\nu}$ is the set of predictable processes $l$, i.e. measurable 
    \[l:([0,T]\times \Omega \times \R^*,\mathcal{P}\otimes \mathcal{B}(\R^*))\to (\R^d, \mathcal{B}(\R^d)); \quad (\omega,t,u)\mapsto l_t(\omega,u)\]
    such that $\|l\|_{\mathcal{H}^{p,d}_{\nu}}^p:=\mathbb{E}\left[\left(\int_0^T\sum_{j=1}^d|l^j_t|^2_{\nu}dt\right)^{\frac{p}{2}}\right]<\infty$.  For $d=1$, we denote  $\mathcal{H}^{p,1}_{\nu}:=\mathcal{H}^{p}_{\nu}$.
    
    \item $\mathcal{P}_p(\R)$ is the set of probability measures on $\R$ with finite $p$th moment. We equip the space $\mathcal{P}_p(\R)$ with the $p$-Wasserstein distance denoted by $\mathcal{W}_p$ and defined as
    \begin{align*}
        \mathcal{W}_p(\mu, \nu) := \inf \left\{\int_{\R \times \R} |x - y|^p \pi(dx, dy) \right\}^{1/p},
    \end{align*}
    where the infimum is over probability measures $\pi \in\mathcal{P}_p(\R\times \R)$ with first and second marginals $\mu$ and $\nu$, respectively.

\end{itemize}

\section{Existence and uniqueness of the solution of mean-field reflected BSDE with jumps}
In this section, we study the well-posedness of a new class of reflected BSDE of mean-field type with jumps, which have the additional feature that the obstacle might depend on the solution too. We start with the following definition.
\begin{Definition}
We say that the quadruple of progressively measurable processes $(Y_t, Z_t, U_t, K_t)_{t\leq T}$ is a solution of the mean-field reflected BSDE associated with $(f,\xi, h)$ if
\begin{align}\label{BSDE1}
    \begin{cases}
    Y \in \mathcal{S}^p,  Z\in \mathcal{H}^{p,1}, U\in \mathcal{H}^{p}_{\nu} \text{ and }  K \in \mathcal{S}_i^p \text{and predictable}, \\
    Y_t = \xi +\int_t^T f(s,Y_{s}, Z_{s}, U_{s}, \P_{Y_s})ds + K_T - K_t-\int_t^T Z_s dB_s - \int_t^T \int_{R^*} U_s(e) \tilde N(ds,de), \,\,\, t \in [0,T],\\
     Y_{t}\geq h(t,Y_{t},\P_{Y_t}),\,\, t \in [0,T], \\
    \int_0^T (Y_{t^-}-h(t^-,Y_{t^-},\P_{Y_{t^-}}))dK_t = 0. \qquad(\text{Skorohod's flatness condition})
    \end{cases}
\end{align}
\end{Definition}
\begin{Remark}
The flatness condition is equivalent to the following condition: if $K = K^c + K^d$, where $K^c$ (resp. $K^d$) represents the continuous (resp. the discontinuous) part of $K$, then
\begin{eqnarray*}
&&\int_0^T (Y_{t}-h(t,Y_{t},\P_{Y_t}))dK^c_t = 0, \text{ and } \forall \tau \in \mathcal{T}_T \text{ predictable}, \\
&&\Delta Y_{\tau} = Y_{\tau}- Y_{\tau^-}= -\Delta K_{\tau}^d = -(h(\tau^-, Y_{\tau^-}, \P_{Y_{t^-}|t=\tau})-Y_{\tau})^+ 1_{\{Y_{\tau^-} = h(\tau^-, Y_{\tau^-}, \P_{Y_{t-}|t=\tau})\}}.
\end{eqnarray*}
\end{Remark}

Throughout the paper, we assume $p\ge 2$ and make the following assumption on $(f,h,\xi)$.
\begin{Assumption} \label{generalAssump} 
The coefficients $f,h$ and $\xi$ satisfy
\begin{itemize} 
    \item[(i)] $f$ is a mapping from $[0,T] \times \Omega \times \R \times \R  \times L^p_{\nu} \times \mathcal{P}_p(\R)$ into $\R$ such that
    \begin{itemize}
        \item[(a)] the process $(f(t,0,0,0,\delta_0))_{t\leq T}$ is $\mathcal{P}$-measurable and belongs to $\mathcal{H}^{p,1}$;
        \item[(b)] $f$ is Lipschitz w.r.t. $(y,z,u,\mu)$ uniformly in $(t,\omega)$, i.e. there exists a positive constant $C_f$ such that $\mathbb{P}$-a.s. for all $t\in [0,T],$ 
        \[|f(t,y_1,z_1,u_1,\mu_1)-f(t,y_2,z_2,u_2,\mu_2)|\leq C_f(|y_1-y_2|+|z_1-z_2|+|u_1-u_2|_{\nu}+\mathcal{W}_p(\mu_1, \mu_2))\]
        for any $y_1,y_2\in \R,$ $z_1, z_2 \in \R$, $u_1, u_2 \in L^p_{\nu}$ and $\mu_1, \mu_2 \in \mathcal{P}_p(\R) $.
        \item[(c)] Assume that $d\mathbb{P} \otimes dt$ a.e. for each $(y,z,u_1,u_2,\mu) \in \R^2 \times (L^2_\nu)^2 \times \mathcal{P}_p(\R),$
        \begin{align}
            f(t,y,z,u_1,\mu)-f(t,y,z,u_2,\mu) \geq \langle\gamma_t^{y,z,u_1,u_2,\mu},l_1-l_2\rangle_\nu,
        \end{align}
        with 
        \begin{align*}
        \gamma:[0,T] \times \Omega \times \R^2 \times (L^2_\nu)^2 \times \mathcal{P}_p(\R) \mapsto L^2_\nu;\\
        (\omega,t,y,z,u_1,u_2,\mu) \mapsto \gamma_t^{y,z,u_1,u_2,\mu}(\omega,\cdot)
        \end{align*}
        $\mathcal{P} \otimes \mathcal{B}(\R^2) \otimes \mathcal{B}((L^2_\nu)^2) \otimes \mathcal{B}(\mathcal{P}_p(\R))$ measurable satisfying $\|\gamma_t^{y,z,u_1,u_2,\mu}(\cdot)\|_\nu \leq C$ for all\\ $(y,z,u_1,u_2,\mu) \in \R^2 \times (L^2_\nu)^2 \times \mathcal{P}_p(\R)$, $d\mathbb{P} \otimes dt$-a.e., where $C$ is a positive constant, and such that $\gamma_t^{y,z,u_1,u_2,\mu}(e) \geq -1$, for all $(y,z,u_1,u_2,\mu) \in \R^2 \times (L^2_\nu)^2 \times \mathcal{P}_p(\R)$, $d\mathbb{P} \otimes dt \otimes d\nu(e)$-a.e.
        \end{itemize}
    \item[(ii)] $h$ is a mapping from $[0,T] \times \Omega \times \R \times \mathcal{P}_p(\R)$ into $\R$ such that
    \begin{itemize}
        \item[(a)] For all $(y,\mu) \in \R \times \mathcal{P}_p(\R)$, $(h(t,y,\mu))_t$ is a right-continuous left-limited process;
        \item[(b)] the process $\left(\sup_{(y,\mu)\in \R\times\mathcal{P}_p(\R)}|h(t,y,\mu)|\right) _{0\le t\leq T}$ belongs to $\mathcal{S}^p$;
        \item[(c)] $h$ is Lipschitz w.r.t. $(y,\mu)$ uniformly in $(t,\omega)$, i.e. there exists two positive constants $\gamma_1$ and $\gamma_2$ such that $\mathbb{P}$-a.s. for all $t\in [0,T],$ 
        \[|h(t,y_1,\mu_1)-h(t,y_2,\mu_2)|\leq \gamma_1|y_1-y_2|+\gamma_2\mathcal{W}_p(\mu_1, \mu_2)\]
        for any $y_1,y_2\in \R$ and $\mu_1,\mu_2 \in \mathcal{P}_p(\R)$.
    \end{itemize}
    \item[(iii)] $\xi \in L^p(\Fc_T)$ and satisfies $\xi \geq h(T,\xi,\P_{\xi})$.
    \end{itemize}
\end{Assumption}

\begin{Remark}
If the barrier $h$ depends only on $\mu$ i.e. $h(t,y,\mu)=h(t,\mu)$, the domination condition (ii)(b) is not needed to derive the results below and can be dropped.
\end{Remark}

\medskip

We now prove the existence and the uniqueness of the solution for the mean-field reflected BSDE with jumps \eqref{BSDE1}, for which the law of the $Y$-component of the solution enters both the driver and the obstacle. Our method relies on Peng's nonlinear expectation (see e.g. Peng \cite{Peng04}) and a characterization of the solution of a mean-field reflected BSDE with jumps as a specific \textit{mean-field nonlinear optimal stopping problem}. This approach is completely different from the one used in \cite{deh19} (where the penalization technique under a domination condition has been used in the case of a Brownian filtration) and allows to weaken the assumptions on the coefficients $f$ and $h$, in particular we get rid of the monotonicity assumptions with respect to the measure component $\mu$.

\medskip
\noindent Let 
$\widehat{\Phi}: \,\mathbb{L}^p_\beta\longrightarrow \mathbb{L}^p_\beta$ be the mapping that associates to a process $Y$ another process $\widehat{\Phi}(Y)$ defined by:
\begin{equation}\label{eq11}
\begin{array}{lll} 
\widehat{\Phi}(Y)_t := \underset{\tau\in\mathcal{T}_t}{\esssup\,} \mathcal{E}^{f \circ Y}_{t,\tau}\left[ h(\tau,Y_{\tau},\P_{Y_s|s=\tau})\ind_{\{\tau<T\}}+\xi \ind_{\{\tau=T\}}\right],
\end{array}
\end{equation}
where $(f \circ Y)(t,y,z,u)=f(t,y,z,u,\P_{Y_t})$. 

\medskip
\begin{Theorem}\label{mf-case}
Suppose that Assumption \ref{generalAssump} is in force for some $p\ge 2$. Assume that $\gamma_1$ and $\gamma_2$ satisfy 
\begin{align} \label{smallnessCond}
\gamma_1^p+\gamma_2^p<2^{2-\frac{3p}{2}}.
\end{align}

Then the system (\ref{BSDE1}) has a unique solution in $\mathcal{S}^p \times \mathcal{H}^{p,1} \times \mathcal{H}^p_\nu \times \mathcal{S}_{i}^p$.
\end{Theorem}

\begin{proof} The proof is organized in three steps.\\
\noindent \uline{\textit{Step 1.}}
We first show that $\widehat{\Phi}$  is a well-defined map from $\mathbb{L}^{p}_{\beta}$ to itself. Indeed, let $\bar{Y} \in \mathbb{L}^{p}_{\beta}$. Since $h$ satisfies Assumption \ref{generalAssump} (ii), it follows that $h(t,\bar{Y}_t,\mathbb{P}_{\bar{Y}_t}) \in \mathcal{S}^p$. Therefore, there exists an unique solution $(\hat{Y}, \hat{Z}, \hat{U}, \hat{K}) \in \mathcal{S}^p \times \mathcal{H}^{p,1} \times \mathcal{H}^p_\nu \times \mathcal{S}_{i}^p$ of the reflected BSDE associated with obstacle $h(t,\bar{Y}_t,\mathbb{P}_{\bar{Y}_t})$, terminal condition $\xi$ and driver $(f \circ {\bar{Y}})(t,\omega,y,z,u)$. Since $f$ satisfies Assumption \ref{generalAssump} (i.c), by classical results on the link between the $Y$-component of the solution of a reflected BSDE and optimal stopping with nonlinear expectations (see e.g. \cite{dqs15}, \cite{dqs16}), we get that $\widehat{\Phi}(\bar{Y})=\hat{Y} \in \mathcal{S}^p \subset \mathbb{L}^{p}_{\beta}$.

\vspace{2mm}

\noindent \uline{\textit{Step 2.}} We show that $\widehat{\Phi}: \mathbb{L}_{\beta}^p \longrightarrow \mathbb{L}_{\beta}^p$ is a contraction on the time interval $[T-\delta,T]$.

\medskip
First, note that by the Lipschitz continuity of $f$ and $h$, for $Y,\bar{Y} \in \mathcal{S}^{p}_\beta$, $Z,\bar{Z} \in \mathcal{H}^{p,1}$, and $U,\bar{U} \in \mathcal{H}_\nu^{p}$,
 \begin{equation}
\begin{array}{lll}\label{f-h-lip-3}
|f(s,Y_s,Z_s,U_s, \P_{Y_s})-f(s,\bar{Y}_s,\bar{Z}_s,\bar{U}_s,\P_{\bar{Y}_s})|
\le  C_f(|Y_s-\bar{Y}_s|+|Z_s-\bar{Z}_s|+|U_s-\bar{U}_s|_{\nu}+\mathcal{W}_p(\P_{Y_s},\P_{\bar{Y}_s})), \\ \\
|h(s,Y_s,\P_{Y_s})-h(s,\bar{Y}_s,\P_{\bar{Y}_s})|\le \gamma_1|Y_s-\bar{Y}_s|+\gamma_2 \mathcal{W}_p(\P_{Y_s},\P_{\bar{Y}_s}).
\end{array}
\end{equation}
For the $p$-Wasserstein distance, we have the following inequality: for $0\leq s \leq u \leq t \leq T$,
\be \label{Wass-property}
\sup_{u\in [s,t]} \mathcal{W}_p(\P_{Y_u},\P_{\bar{Y}_u}) \leq \sup_{u\in [s,t]}(\mathbb{E}[|Y_u-\bar{Y}_u|^p])^{1/p},
\ee
from which we derive the following useful inequality
\be
\sup_{u\in [s,t]} \mathcal{W}_p(\P_{Y_u},\delta_0) \leq \sup_{u\in [s,t]}(\mathbb{E}[|Y_u|^p])^{1/p}.
\ee
Fix $Y,\hat{Y} \in \mathbb{L}_\beta^p$. For any $t\leq T$, by the estimates (A.1) on BSDEs, we have  
\begin{equation*}\begin{array} {lll}
\quad |\widehat{\Phi}(Y)_t-\widehat{\Phi}(\bar{Y})_t|^{p} \\
=|\underset{\tau\in\mathcal{T}_t}{\esssup\,}\mathcal{E}_{t,\tau}^{f \circ Y }[h(\tau,Y_\tau,\P_{Y_s|s=\tau})\ind_{\{\tau<T\}}+\xi\ind_{\{\tau=T\}}]
 -\underset{\tau\in\mathcal{T}_t}{\esssup\,}\mathcal{E}_{t,\tau}^{f \circ \bar{Y} }[h(\tau,\bar{Y}_\tau,\P_{\bar{Y}_s|s=\tau})\ind_{\{\tau<T\}}+\xi \ind_{\{\tau=T\}}]|^{p}
\\  \le \underset{\tau\in\mathcal{T}_t}{\esssup\,}\left|\mathcal{E}_{t,\tau}^{f \circ Y}[h(\tau,Y_\tau,\P_{Y_s|s=\tau})\ind_{\{\tau<T\}}+\xi\ind_{\{\tau=T\}}]- \mathcal{E}_{t,\tau}^{f \circ \bar{Y} }[h(\tau,\bar{Y}_\tau,\P_{\bar{Y}_s|s=\tau})\ind_{\{\tau<T\}}+\xi\ind_{\{\tau=T\}}] \right|^{p}\\
\le \underset{\tau\in\mathcal{T}_t}{\esssup\,}2^{\frac{p}{2}-1}\E[\eta^p \left(\int_t^\tau e^{2 \beta  (s-t)} |(f \circ Y) (s,\widehat{Y}^{\tau}_s,\widehat{Z}^{\tau}_s,\widehat{U}^{\tau}_s)-(f \circ \bar{Y}) (s,\widehat{Y}^{\tau}_s, \widehat{Z}^{\tau}_s,\widehat{U}^{\tau}_s)|^2 ds\right)^{\frac{p}{2}} \\ \qquad\qquad \qquad +e^{p \beta(\tau-t)}|h(\tau,Y_\tau,\P_{Y_s|s=\tau})-h(\tau,\bar{Y}_\tau, \P_{\bar{Y}_s|s=\tau})|^p
|\mathcal{F}_t]\\ = \underset{\tau\in\mathcal{T}_t}{\esssup\,} 2^{\frac{p}{2}-1}\E[\eta^p \left(\int_t^\tau e^{2 \beta (s-t)}|f(s,\widehat{Y}^{\tau}_s,\widehat{Z}^{\tau}_s,\widehat{U}^{\tau}_s,
\P_{Y_s})-f(s,\widehat{Y}^{\tau}_s,\widehat{Z}^{\tau}_s,\widehat{U}^{\tau}_s, \P_{\bar{Y}_s})|^2 ds \right)^{\frac{p}{2}}  \\  \qquad\qquad \qquad +e^{p \beta(\tau-t)}|h(\tau,Y_\tau,\P_{Y_s|s=\tau})-h(\tau,\bar{Y}_\tau, \P_{\bar{Y}_s|s=\tau})|^p
|\mathcal{F}_t],
 \end{array}
\end{equation*}
with $\eta$, $\beta>0$ such that  $\eta \leq \frac{1}{C_f^2}$ and  $\beta \geq 2 C_f+\frac{3}{\eta}$, where $(\widehat{Y}^{\tau},\widehat{Z}^{\tau},\widehat{U}^{\tau})$ is the solution of the BSDE associated with driver $f\circ \bar{Y}$ and terminal condition $h(\tau,{\bar{Y}}_\tau,\P_{\bar{Y}_s|s=\tau})\ind_{\{\tau<T\}}+\xi\ind_{\{\tau=T\}}$.
Therefore, using \eqref{f-h-lip-3} and the fact that 
\begin{equation} \label{Wass-property-1}\mathcal{W}_p^{p} (\P_{Y_s|s=\tau},\P_{\bar{Y}_s|s=\tau})\le \E[|Y_s-\bar{Y}_s|^p]_{|s=\tau}\le \underset{\tau\in \mathcal{T}_t}{\sup}\E[|Y_{\tau}-\bar{Y}_{\tau}|^p],
\end{equation}
we have, for any $t\in [T-\delta,T]$,
\begin{equation}\label{ineq-evsn-3}
\begin{array} {ll}
e^{p \beta t}|\widehat{\Phi}(Y)_t-\widehat{\Phi}(\bar{Y})_t|^p \le \underset{\tau\in\mathcal{T}_t}{\esssup\,} 2^{\frac{p}{2}-1}\E[\int_t^\tau \delta^{\frac{p-2}{2}}e^{p \beta s}\eta^pC_f^p\E[|Y_s-\bar{Y}_s|^p]ds  \\  \qquad\qquad \qquad \qquad\qquad\qquad\qquad +e^{p \beta \tau}\left(\gamma_1|Y_\tau-\bar{Y}_\tau|+\gamma_2\left\{\E[|Y_s-\bar{Y}_s|^p]_{|s=\tau}\right\}^{1/p}\right)^p
|\mathcal{F}_t],
\\
\qquad\qquad \qquad \qquad\qquad \le \underset{\tau\in\mathcal{T}_t}{\esssup\,} 2^{\frac{p}{2}-1}\E[\int_t^\tau \delta^{\frac{p-2}{2}} e^{p \beta s}\eta^pC_f^p\E[|Y_s-\bar{Y}_s|^p]ds  \\  \qquad\qquad \qquad \qquad\qquad\qquad\qquad +e^{p \beta\tau}\left\{2^{p-1}\gamma_1^p|Y_\tau-\bar{Y}_\tau|^p+2^{p-1}\gamma_2^p\E[|Y_s-\bar{Y}_s|^p]_{|s=\tau}\right\}
|\mathcal{F}_t].
\end{array}
\end{equation}
Therefore,
$$
e^{p \beta t}|\widehat{\Phi}(Y)_t-\widehat{\Phi}(\bar{Y})_t|^p \le \underset{\tau\in\mathcal{T}_t}{\esssup\,} \E[G(\tau)|\mathcal{F}_t]:=V_t,
$$
where
$$
G(\tau):=2^{\frac{p}{2}-1} \left(\int_{T-\delta}^{\tau} \delta^{\frac{p-2}{2}} e^{p \beta s}\eta^pC_f^p\E[|Y_s-\bar{Y}_s|^p]ds+e^{p \beta \tau}(2^{p-1}\gamma_1^p|Y_{\tau}-\bar{Y}_{\tau}|^p+2^{p-1}\gamma_2^p\E[|Y_{s}-\bar{Y}_{s}|^p]_{|s=\tau})\right).
$$
which yields
\begin{equation}\label{est-1}
\underset{\tau\in\mathcal{T}_{T-\delta}}{\sup}\E[e^{p \beta \tau}|\widehat{\Phi}(Y)_{\tau}-\widehat{\Phi}(\bar{Y})_{\tau}|^p] \le \underset{\tau\in\mathcal{T}_{T-\delta}}{\sup}\E[V_{\tau}].
\end{equation}
We have
\begin{equation}\label{est-2}
\underset{\tau\in\mathcal{T}_{T-\delta}}{\sup}\E[V_{\tau}]\le \alpha \underset{\tau\in\mathcal{T}_{T-\delta}}{\sup}\E[e^{p \beta \tau}|Y_{\tau}-\bar{Y}_{\tau}|^p]. 
\end{equation}
Indeed, by Lemma D.1 in \cite{KS98}, for any $\sigma\in\mathcal{T}_{T-\delta}$,  there exists a sequence $(\tau_n)_n$ of stopping times in $\mathcal{T}_{\sigma}$ such that
$$
V_{\sigma}=\underset{n\to\infty}{\lim}\E[G(\tau_n)|\mathcal{F}_{\sigma}]
$$
and so, by Fatou's Lemma, we have
$$
\E[V_{\sigma}]\le \underset{n\to\infty}{\underline{\lim}}\E[G(\tau_n)]\le \underset{\tau\in\mathcal{T}_{T-\delta}}{\sup}\E[G(\tau)].
$$
Therefore,
$$
\underset{\tau\in\mathcal{T}_{T-\delta}}{\sup}\E[e^{p \beta \tau}|\widehat{\Phi}(Y)_\tau-\widehat{\Phi}(\bar{Y})_\tau|^p] \le \underset{\tau\in\mathcal{T}_{T-\delta}}{\sup}\E[G(\tau)].
$$
Using \eqref{Wass-property-1} and noting that $e^{p\beta\tau}\E[|Y_s-\bar{Y}_s|^p]_{|s=\tau}=\E[e^{p\beta s}|Y_s-\bar{Y}_s|^p]_{|s=\tau}$, we obtain
$$
\underset{\tau\in\mathcal{T}_{T-\delta}}{\sup}\E[G(\tau)]\le \alpha \underset{\tau\in\mathcal{T}_{T-\delta}}{\sup}\E[e^{p \beta \tau}|Y_{\tau}-\bar{Y}_{\tau}|^p]
$$
where $\alpha:=2^{\frac{p}{2}-1}\delta^{\frac{p}{2}} \eta^p C_f^p+2^{\frac{p}{2}-1} 2^{p-1}(\gamma_1^p+\gamma_2^p)$. This in turn yields \eqref{est-2}.

Assuming $(\gamma_1,\gamma_2)$ satisfies
\begin{equation*}
\gamma_1^p+\gamma_2^p<2^{2-\frac{3p}{2}}
\end{equation*}
we can choose  
$$
0<\delta<\left(\frac{1}{2^{\frac{p}{2}-1}\eta^p C_f^p}
\left(1-2^{\frac{3p}{2}-1}(\gamma^p_1+\gamma^p_2)\right)\right)^{2/p}
$$ 
 to make $\widehat{\Phi}$ a contraction on $\mathbb{L}^p_\beta$ over the time interval $[T-\delta,T]$, i.e. $\widehat{\Phi}$ admits a unique fixed point over $[T-\delta,T]$.
 
\medskip
\noindent \uline{\textit{Step 3.}} We finally show that the mean-field RBSDE with jumps \eqref{BSDE1} has a unique solution.

\medskip
Given $Y\in \mathbb{L}_\beta^{p}$ the solution over $[T-\delta,T]$ obtained in \textit{Step 2}, let $(\hat{Y},\hat{Z},\hat{U},\hat{K}) \in \mathcal{S}^p \times \mathcal{H}^{p,1} \times \mathcal{H}^p_\nu \times \mathcal{S}_{i}^p$ be the unique solution of the standard reflected BSDE, with barrier $h(s,Y_s,\P_{Y_s})$ (which belongs to $\mathcal{S}^{p}$ by Assumption \ref{generalAssump} (ii)(a)) and driver $g(s,y,z,u):=f(s,y,z,u,\P_{Y_s})$. Then,
$$
\hat{Y}_t=\underset{\tau\in\mathcal{T}_t}{\esssup\,}\mathcal{E}_{t,\tau}^{g}\left[h(\tau,Y_{\tau},\P_{Y_s}|_{s=\tau})\ind_{\{\tau<T\}}+\xi \ind_{\{\tau=T\}}\right].
$$
By the fixed point argument from Step 2, we have $\hat{Y}=Y$ a.s., which implies that $Y \in \mathcal{S}^{p}$ and
$$
Y_t = \xi +\int_t^T f(s,Y_{s},  \hat{Z}_s,\hat{U}_s, \P_{Y_s})ds + \hat{K}_T - \hat{K}_t -\int_t^T \hat{Z}_s dB_s - \int_t^T \int_{R^*} \hat{U}_s(e) \tilde N(ds,de), \quad T-\delta \leq t \leq T,
$$
which yields existence of a solution.
Furthermore, the process $Y_t$ is a \textit{nonlinear} g-supermartingale. Hence, by the uniqueness of the  \textit{nonlinear} Doob-Meyer decomposition, we obtain uniqueness of the associated processes $(\hat{Z},\hat{U},\hat{K})$. This yields existence and uniqueness of \eqref{BSDE1} in this general case.

Applying the same reasoning on each time interval $[T-(i+1)\delta,T-i\delta]$ with a similar dynamics but terminal condition $Y_{T-i\delta}$ at time $T-i\delta$, we build recursively for $i=1$ to any $n$ a solution $(Y,Z,U,K)\in \mathcal{S}^p \times \mathcal{H}^{p,1} \times \mathcal{H}^{p}_{\nu} \times S_i^p$ on each time interval $[T-(i+1)\delta,T-i\delta].$ Pasting these processes, we derive a unique solution $(Y,Z,U,K)$ satisfying (\ref{BSDE1}) on the full time interval $[0,T]$.
\qed
\end{proof}

\medskip
Suppose that the assumptions of Theorem \ref{mf-case} hold. We now give a sufficient condition on the barrier $h$ under which the component $(Y_t)$ of the solution $(Y,Z,U,K)$ of the mean-field reflected BSDE \eqref{BSDE1} has only totally inaccessible jumps, i.e. the process $(K_t)$ is continuous.
We first recall the following definition (see e.g. \cite{dqs15}, \cite{dqs16}).
\begin{Definition}
A progressive process $(\phi_t)$ is called \textit{left-upper semicontinuous (in short l.u.s.c) along stopping times}  if for all $\tau \in \mathcal{T}_0$ and for each non-decreasing sequence of stopping times $(\tau_n)$ such that $\tau_n \uparrow \tau$ a.s., $\phi_\tau \geq \underset{n \mapsto \infty}{\limsup}\,\phi_{\tau_n}$ a.s.
\end{Definition}
Note that no condition is required at totally inaccessible jumps. When the process $(\phi_t)$ is left limited, $(\phi_t)$ is l.u.s.c. along stopping times if and only if for each predictable stopping time $\tau \in \mathcal{T}_0$, $\phi_\tau \geq \phi_{\tau^-}$ a.s.
\begin{Theorem}\label{inacces}
Suppose that $\gamma_1$ and $\gamma_2$ satisfy the condition \eqref{smallnessCond}. Assume that $h$ takes the form $h(t,\omega,y,\mu):=\xi_t(\omega)+\kappa(y,\mu)$, where $\xi$  belongs to $\mathcal{S}^p$ and is a left upper semicontinuous process along stopping times, and $\kappa$ is a bounded and Lipschitz function with respect to $(y,\mu)$. Let $(Y,Z,U,K)$ be the unique solution of the mean-field reflected BSDE \eqref{BSDE1}. Then $(Y_t)$ has jumps only at totally inaccessible stopping times (i.e. the predictable process $(K_t)$ is continuous).
\end{Theorem}
\begin{proof}
Let $Y^0 \equiv 0$. Define $Y_t^{n+1}:=\widehat{\Phi}(Y^n)_t$, where the operator $\widehat{\Phi}$ is defined by \eqref{eq11}. Since the function $\kappa$ is continuous with respect to $(y,\mu)$ and $\xi_t$ is left upper semicontinuous along stopping times, it follows that the process $h(t,Y_t^0,\mathbb{P}_{Y_t^0})$ is also left upper semicontinuous along stopping times.  Then $Y^1$ corresponds to the first component of the reflected BSDE with driver $f \circ Y^0$, terminal condition $\xi$ and obstacle process $h(t,Y_t^0,\mathbb{P}_{Y_t^0})$. Since $h(t,Y_t^0,\mathbb{P}_{Y_t^0})$ is left upper semicontinuous along stopping times, it follows that $Y^1$ has jumps only at totally inaccessible jumps and thus the map $t \mapsto \mathbb{P}_{Y_t^0}$ is continuous. We deduce  that $h(t,Y_t^1,\mathbb{P}_{Y_t^1})$ is left upper semicontinuous along stopping times, and by the same arguments as above, $Y^2$ has jumps only at totally inaccessible jumps too. By induction, one can show that $Y_t^{n}$ has jumps only at totally inaccessible jumps for all $n \geq 1$.\\
Since the condition $\eqref{smallnessCond}$ is satisfied, by Theorem \ref{mf-case}, the sequence $Y_t^n$ is Cauchy for the norm $\mathbb{L}^p$ and thus converges in $\mathbb{L}^p$ to the component $Y$ of the unique solution of the mean-field reflected BSDE \eqref{BSDE1}.\\
Let $\tau \in \mathcal{T}_0$ be a predictable stopping time. Since $\Delta Y^{n}_\tau=0$ a.s. for all $n$, we obtain
\begin{align}
    \mathbb{E}\left[|\Delta Y_\tau|^p\right]=\mathbb{E}\left[|\Delta Y_\tau-\Delta Y^{n}_\tau|^p\right] \leq 2^p \sup_{\tau \in \mathcal{T}_0}\mathbb{E}\left[|Y_\tau-Y^{n}_\tau|^p\right],
\end{align}
which implies that $\Delta Y_\tau=0$ a.s. By the arbitrariness of the predictable stopping time $\tau \in \mathcal{T}_0$, the result follows.
\end{proof}

\section{Weakly interacting nonlinear Snell envelopes}
We will now discuss the interpretation of the mean-field reflected BSDE with jumps studied in the previous section at the particle level and study the well-posedness of the associated particle system.  To this end, for a given vector $\textbf{x} := (x^1, \ldots, x^n) \in \R^n$, denote the empirical measure associated to $\bf x$ by
$$
L_n [\textbf{x}] := \frac{1}{n} \sum_{k=1}^n \delta_{x^k}.
$$
Let $\{B^i\}_{1\leq i \leq n}$, $\{\tilde N^i\}_{1\leq i \leq n}$ be independent copies of $B$ and $\tilde N$ and denote by $\mathbb{F}^n:=\{\mathcal{F}_t^n\}_{t \in [0,T]}$ the completion of the filtration generated by $\{B^i\}_{1\leq i \leq n}$ and $\{\tilde N^i\}_{1\leq i \leq n}$. Thus, for each $1\le i\le n$, the filtration generated by $(B^i,\tilde{N}^i)$ is a sub-filtration of $\mathbb{F}^n$, and we denote by $\widetilde{\mathbb{F}}^i:=\{\widetilde{\mathcal{F}}^i_t\}$ its completion. Let $\mathcal{T}^n_t$ be the set of $\mathbb{F}^n$ stopping times with values in $[0,T]$. 

Consider a family of weakly interacting processes $\textbf{Y}^n := (Y^{1,n},\ldots, Y^{n,n})$ evolving backward in time as follows: for $i= 1,\ldots,n$,
\begin{align}\label{BSDEParticle}
    \begin{cases}
    Y^{i,n}_t = \xi^{i,n} +\int_t^T f(s,Y^{i,n}_{s},Z^{i,i,n}_{s},U^{i,i,n}_{s}, L_n[\textbf{Y}^n_{s}])ds + K^{i,n}_T - K^{i,n}_t  \\
    	\quad	\quad \quad  - \int_t^T \sum_{j=1}^nZ^{i,j,n}_s dB^j_s -  \int_t^T \int_{R^*} \sum_{j=1}^n U^{i,j,n}_s(e) \tilde N^j(ds,de), \quad  t \in [0,T],\\
    Y^{i,n}_{t} \geq h(t,Y^{i,n}_{t},L_n[\textbf{Y}^n_{t}]),  \quad \forall t \in [0,T], \\
    \int_0^T (Y^{i,n}_{t^-}-h(t^-,Y^{i,n}_{t^-},L_n[\textbf{Y}^n_{t^-}]))dK^{i,n}_t = 0. 
    \end{cases}
\end{align}

\begin{Assumption}\label{Assump-PS} We assume that
\begin{itemize} 
\item $f$ and $h$ are deterministic functions of $(t,y,z,u,\mu)$ and $(t,y,\mu)$, respectively. 
\item  $\xi^{i,n}\in L^p(\mathcal{F}_T^n), \quad i=1,\ldots,n$.
\item  $\xi^{i,n}\ge h(T,\xi^{i,n},L_n[{\bf \xi}^n])\,\, \text{a.s.}\,\,\, i=1,\ldots,n.$
\end{itemize}
\end{Assumption}
For a more general form of the driver $f$ and the obstacle $h$ see Remark \ref{omega-f-h}, below. 

\medskip

\begin{Remark} $ $
\begin{itemize}
    \item[$(1)$] Note that in the system  \eqref{BSDEParticle} we need to use all the noise processes $((B^i,\tilde{N}^i),\,i=1,\ldots,n)$ to describe the dynamics $Y^{i,n}$ of each particle $i$. This is due to the fact that the construction of each individual backward SDE is based on the martingale representation theorem which should involve all the noise processes $((B^i,\tilde{N}^i),\,i=1,\ldots,n)$ and not only individual noise, as is the case for forward SDEs.
    \item[$(2)$] The system \eqref{BSDEParticle} is a 'true' system of weakly interacting processes, and is different from the construction in \cite{bdlp09} and \cite{l14} where the MF-BSDE is obtained a limit of {\em one} BSDE indexed by $n$, driven by one fixed Brownian motion and whose driver is an average of the Wiener measure preserving shift operator applied to measurable functions of the Wiener path, for which the strong law of large numbers for i.i.d. sequences holds.  
\end{itemize}
\end{Remark}

\noindent
Note that we have the inequality
\begin{equation}\label{ineq-wass-1}
\mathcal{W}_p^p(L_n[\textbf{x}], L_{n}[\textbf{y}]) \le \frac{1}{n}\sum_{j=1}^{n}|x_j-y_j|^p,
\end{equation}
and, in particular,
\begin{equation}\label{ineq-wass-2}
\mathcal{W}_p^p(L_{n}[\textbf{x}], L_n[\textbf{0}]) \le \frac{1}{n}\sum_{j=1}^{n}|x_j|^p,
\end{equation}
where we note that $L_n[\textbf{0}]=\frac{1}{n}\sum_{j=1}^{n}\delta_0=\delta_0$.

\noindent Endow the product spaces $\spn_\beta:=\mathcal{S}_\beta^p\times \mathcal{S}_\beta^p \times \cdots \times \mathcal{S}_\beta^p$  and 
$\lpn_{\beta}:=\mathbb{L}_{\beta}^p\times \mathbb{L}_{\beta}^p \times \cdots \times \mathbb{L}_{\beta}^p$
with the respective norms
\begin{equation}\label{n-norm-2}
\|h\|^p_{\spn_\beta}:=\sum_{1 \le i\le n}\|h^i\|^p_{\mathcal{S}_\beta^p},\qquad \|h\|^p_{\lpn_\beta}:=\sum_{1 \le i\le n}\|h^i\|^p_{\mathbb{L}_\beta^p}.
\end{equation}
Note that $\spn_\beta$ and $\lpn_\beta$ are complete metric spaces. We denote by $\spn:=\spn_0,\,\, \lpn:=\lpn_0$.\\
\noindent Let 
$\widetilde{\Phi}: \,\lpn_\beta\longrightarrow \lpn_\beta$ to be the mapping that associates to a process $\textbf{Y}^n:=(Y^{1,n},Y^{2,n},\ldots,Y^{n,n})$ the process $\widetilde{\Phi}(\textbf{Y}^n)=(\widetilde{\Phi}^1(\textbf{Y}^n),\widetilde{\Phi}^2(\textbf{Y}^n),\ldots, \widetilde{\Phi}^n(\textbf{Y}^n))$ defined by the following system: for every $i=1,\ldots n$ and $t\le T$,
\begin{equation}\label{snell-i-2}
\begin{array}{lll} 
\widetilde{\Phi}^i(\Ybf^n)_t := \underset{\tau\in\mathcal{T}^n_t}{\esssup\,} \mathcal{E}^{\textbf{f}^{i} \circ \textbf{Y}^n}_{t,\tau}\left[ h(\tau,Y^{i,n}_{\tau},L_{n}[\textbf{Y}^n_{s}])_{s=\tau})1_{\{\tau<T\}}+\xi^{i,n}1_{\{\tau=T\}}\right],
\end{array}
\end{equation}
where 
$$
\begin{array}{lll}
\textbf{f}^{i} \circ \textbf{Y}^n:[0,T] \times \Omega \times \R \times \R^n \times (\R^*)^n \mapsto \R \\  \qquad\qquad (\textbf{f}^{i} \circ \textbf{Y}^n)(t,\omega, y,z,u)=f(t,\omega, y,z^i,u^i,L_{n}[\textbf{Y}^n_{t}](\omega)), \quad i=1,\ldots,n.
\end{array}
$$

\begin{Theorem}\label{existParticle-2}
Suppose that Assumptions \ref{generalAssump} and \ref{Assump-PS} are in force for some $p\ge 2$. Suppose further that $\gamma_1$ and $\gamma_2$ satisfy \eqref{smallnessCond} i.e.
\begin{align*}\label{smallnessCond-new}
\gamma_1^p+\gamma_2^p<2^{2-\frac{3p}{2}}.
\end{align*}

Then the system (\ref{BSDEParticle}) has a unique solution in $\mathcal{S}^{p,\otimes n} \otimes \mathcal{H}^{p,n\otimes n} \otimes \mathcal{H}_{\nu}^{p,n\otimes n} \otimes \mathcal{S}_{i}^{p,\otimes n}$.
\end{Theorem}

\medskip

\begin{proof}
\uline{\textit{Step 1.}} We first show that $\widetilde{\Phi}$  is a well-defined map from $\lpn_{\beta}$ to itself. To this end, we linearize the mapping $h$ as follows: for $i=1,\dots, n$ and $0 \leq s\le T$,
$$
\begin{array}{lll}
h(s,Y^{i,n}_s,L_{n}[\textbf{Y}^n_s])=h(s,0,L_{n}[\textbf{0}])+a^i_h(s)Y^{i,n}_s+b^i_h(s)\mathcal{W}_p(L_{n}[\textbf{Y}^n_s], L_{n}[\textbf{0}]),
\end{array}
$$
where $a^i_h(\cdot)$ and $b^i_h(\cdot)$ are adapted processes given by 
\begin{equation}
\left\{\begin{array}{lll}
a^i_h(s):=\frac{h(s,Y^{i,n}_s,L_{n}[\textbf{Y}^n_s])-h(s,0,L_{n}[\textbf{Y}^n_s])}{Y^{i,n}_s}\ind_{\{Y^{i,n}_s\neq 0\}},\\ \\  b^i_h(s):=\frac{h(s,0,L_{n}[\textbf{Y}^n_s])-h(s, 0,L_{n}[\textbf{0}])}{\mathcal{W}_p(L_{n}[\textbf{Y}^n_s], L_{n}[\textbf{0}])}\ind_{\{\mathcal{W}_p(L_{n}[\textbf{Y}^n_s], L_{n}[\textbf{0}])\neq 0\}}.
\end{array}
\right.
\end{equation}
and which, by the Lipschitz continuity of $h$, satisfy
$|a^i_h(\cdot)|\le \gamma_1,\,\,\, |b^i_h(\cdot)|\le \gamma_2.$\\
By the estimates (A.1) on BSDEs established in Proposition \ref{estimates}, below, we obtain, for any stopping time $\tau \in \mathcal{T}^n_t$ the following:
\begin{align}
    &\left|\mathcal{E}_{t,\tau}^{\textbf{f}^{i} \circ \textbf{Y}^n }\left[\xi^{i,n}1_{\{\tau=T\}}+h(\tau,Y^{i,n}_{\tau},L_{n}[\textbf{Y}^n_{s}])_{s=\tau})1_{\{\tau<T\}}\right]\right|^p\nonumber \\
    &\leq 2^{\frac{p}{2}-1}\left(\mathbb{E}\left[e^{p \beta(\tau-t)}\left|\xi^{i,n}1_{\{\tau=T\}}+h(\tau,Y^{i,n}_{\tau},L_{n}[\textbf{Y}^n_{s}])_{s=\tau})1_{\{\tau<T\}}\right|^p|\mathcal{F}_t \right] \right. \nonumber \\
    & \quad \quad \quad+\left. \eta^p \mathbb{E}\left[\left\{\int_t^\tau e^{2 \beta (s-t)}|f|^2(s,0,0,0,L_{n}[\textbf{Y}^n_{s}])ds\right\}^{p/2}|\mathcal{F}_t\right]\right),
\end{align}
with $\eta$, $\beta>0$ such that  $\eta \leq \frac{1}{C_f^2}$ and  $\beta \geq 2 C_f+\frac{3}{\eta}$. Moreover, since $f(t,0,0,0,\delta_0) \in \mathcal{H}^{p,1}$, $\xi^{i,n}\in L^p(\mathcal{F}_T)$, $(h(t,0,\delta_0))_{0\leq t\leq T}\in \mathcal{S}_\beta^p$ and $(Y^{i,n})_{1 \leq i \leq n}\in \lpn_\beta$, the non-negative c{\`a}dl{\`a}g process $(M^{i,\beta}_t)_{0 \leq t\leq T}$ defined by
\begin{align*}
    M^{i,\beta}_t := & \left(e^{\beta t}|h(t,0,L_{n}[\textbf{0}])| + \gamma_1 e^{\beta t} |Y^{i,n}_t| + \gamma_2 e^{\beta t} \mathcal{W}_p(L_{n}[\textbf{Y}^n_t], L_{n}[\textbf{0}])+e^{\beta T}|\xi^{i,n}|\ind_{\{t=T\}}\right)^p \\
    & +\eta^{p} \left(\int_0^t \left\{e^{\beta s}|f(s,0,0,0,L_{n}[\textbf{0}])|+C_f e^{\beta s} \mathcal{W}_p(L_{n}[\textbf{Y}^n_s], L_{n}[\textbf{0}])\right\}^2ds\right)^{p/2}
\end{align*}
belongs to $\mathbb{L}^{1}$. One can easily show that for $i=1,\dots,n,$ it follows that $\widetilde{\Phi}(\textbf{Y}^n) \in \lpn_\beta$.\\
\vspace{2mm}

\uline{\textit{Step 2.}} We now show that $\widetilde{\Phi}$ is a contraction on the time interval $[T-\delta,T]$.

\medskip
Fix $\mathbf{Y}^n=(Y^{1,n},\dots,Y^{n,n}),\bar{\mathbf{Y}}^n=(\bar Y^{1,n},\dots,\bar Y^{n,n}) \in \lpn_\beta$, $(\hat{Y},\tilde{Y}) \in (\mathcal{S}_\beta^{p})^2$, $(\hat{Z},\tilde{Z}) \in (\mathcal{H}^{p,1})^2$, $(\hat{U},\tilde{U}) \in (\mathcal{H}_\nu^{p})^2$. By the Lipschitz continuity of $f$ and $h$, we get
  \begin{eqnarray}\label{f-h-lip-4}
&&|f(s,\hat{Y}_s,\hat{Z}_s,\hat{U}_s, L_{n}[\textbf{Y}^n_s])-f(s,\tilde{Y}_s,\tilde{Z}_s,\tilde {U}_s,L_{n}[\bar{\textbf{Y}}^n_s])| \le  C_f(|\hat{Y}_s-\tilde{Y}_s|+|\hat{Z}_s-\tilde{Z}_s|\nonumber \\
&&\quad \quad\quad \quad \quad \quad\quad \quad\quad \quad\quad \quad\quad \quad\quad \quad \quad\quad \quad\quad \quad\quad\quad \quad +|\hat{U}_s-\tilde{U}_s|_\nu +\mathcal{W}_p(L_{n}[\textbf{Y}^n_s], L_{n}[\bar{\textbf{Y}}^n_s])), \nonumber\\ 
&&|h(s,Y^{i,n}_s,L_{n}[\textbf{Y}^n_s])-h(s,\bar{Y}^{i,n}_s,L_{n}[\bar{\textbf{Y}}^n_s])| \le  \gamma_1|Y^{i,n}_s-\bar{Y}^{i,n}_s|+\gamma_2\mathcal{W}_p(L_{n}[\textbf{Y}^n_s], L_{n}[\bar{\textbf{Y}}^n_s]).
\end{eqnarray}
By \eqref{ineq-wass-1}, we have
\begin{equation}\label{ineq-wass-1-1-3}
\mathcal{W}_p^p(L_{n}[\textbf{Y}^n_s], L_{n}[\bar{\textbf{Y}}^n_s])) \le \frac{1}{n}\sum_{j=1}^n |Y_s^{j,n}-\bar{Y}_s^{j,n}|^p.
\end{equation}
Then, using the estimates (A.1) again, for any $t\leq T$ and $i=1,\ldots, n$, we have  
\begin{equation*}\begin{array} {lll}
\quad |\widetilde{\Phi}^i(\Ybf^n)_t-\widetilde{\Phi}^i(\bar{\Ybf}^n)_t|^{p} \\
=|\underset{\tau\in\mathcal{T}^n_t}{\esssup\,}\mathcal{E}_{t,\tau}^{\textbf{f}^{i} \circ \Ybf^n }[h(\tau,Y^{i,n}_\tau,L_{n}[\Ybf^n_s]_{s=\tau})\ind_{\{\tau<T\}}+\xi^{i,n}\ind_{\{\tau=T\}}]
 \\ \quad -\,\,\underset{\tau\in\mathcal{T}^n_t}{\esssup\,}\mathcal{E}_{t,\tau}^{\textbf{f}^{i} \circ \bar{\Ybf}^n }[h(\tau,\bar{Y}^{i,n}_\tau,L_{n}[\bar{\Ybf}^n_s]_{s=\tau})1_{\{\tau<T\}}+\xi^{i,n}\ind_{\{\tau=T\}}]|^{p}
\\    \le \underset{\tau\in\mathcal{T}^n_t}{\esssup\,}\left|\mathcal{E}_{t,\tau}^{\textbf{f}^{i} \circ \Ybf}[h(\tau,Y^{i,n}_\tau,L_{n}[\Ybf^n_s]_{s=\tau})\ind_{\{\tau<T\}}+\xi^{i,n}\ind_{\{\tau=T\}}] \right. \\ \left. \qquad\qquad\qquad\qquad - \mathcal{E}_{t,\tau}^{\textbf{f}^{i} \circ \bar{\Ybf}^n }[h(\tau,\bar{Y}^{i,n}_\tau,L_{n}[\bar{\Ybf}^n_s]_{s=\tau})1_{\{\tau<T\}}+\xi^{i,n}\ind_{\{\tau=T\}}] \right|^{p}\\
\le \underset{\tau\in\mathcal{T}^n_t}{\esssup\,}2^{p/2-1}\E[\eta^p\left(\int_t^\tau e^{2 \beta (s-t)}  |(\textbf{f}^{i} \circ \Ybf^n) (s,\widehat{Y}^{i,\tau}_s,\widehat{Z}^{i,\tau}_s,\widehat{U}^{i,\tau}_s)-(\textbf{f}^{i} \circ \bar{\Ybf}^n) (s,\widehat{Y}^{i,\tau}_s, \widehat{Z}^{i,\tau}_s,\widehat{U}^{i,\tau}_s)|^2 ds \right)^{p/2}  \\ \qquad\qquad \qquad +e^{p \beta(\tau-t)}|h(\tau,Y^{i,n}_\tau,L_{n}[\Ybf^n_s]_{s=\tau})-h(\tau,\bar{Y}^{i,n}_\tau, L_{n}[\bar{\Ybf}^n_s]_{s=\tau})|^p
|\mathcal{F}_t] \\
\le \underset{\tau\in\mathcal{T}^n_t}{\esssup\,} 2^{p/2-1}\E[\eta^p \left(\int_t^\tau e^{2 \beta (s-t)} |f(s,\widehat{Y}^{i,\tau}_s,\widehat{Z}^{i,i,\tau}_s,\widehat{U}^{i,i,\tau}_s,
L_{n}[\Ybf^n_s])-f(s,\widehat{Y}^{i,\tau}_s,\widehat{Z}^{i,i,\tau}_s,\widehat{U}^{i,i,\tau}_s, L_{n}[\bar{\Ybf}^n_s])|^2 ds \right)^{p/2} \\ \qquad\qquad \qquad +e^{p \beta(\tau-t)}|h(\tau,Y^{i,n}_\tau,L_{n}[\Ybf^n_s]_{s=\tau})-h(\tau,\bar{Y}^{i,n}_\tau, L_{n}[\bar{\Ybf}^n_s]_{s=\tau})|^p
|\mathcal{F}_t],
 \end{array}
\end{equation*}
where $(\widehat{Y}^{i,\tau},\widehat{Z}^{i,\tau},\widehat{U}^{i,\tau})$ is the solution of the BSDE associated with driver $\textbf{f}^{i}\circ \bar{\Ybf}^n$, terminal time $\tau$ and terminal condition $h(\tau,\bar{Y}^{i,n}_\tau,L_{n}[\bar{\Ybf}^n_s]_{s=\tau})\ind_{\{\tau<T\}}+\xi^{i,n}\ind_{\{\tau=T\}}$.

\noindent Therefore, using \eqref{f-h-lip-4} and \eqref{ineq-wass-1-1-3}, we have, for any $t\le T$ and $i=1,\ldots, n$,
\begin{equation}\label{ineq-evsn-2}\begin{array} {ll}
e^{p \beta t}|\widetilde{\Phi}^i(\Ybf^n)_t-\widetilde{\Phi}^i(\bar{\Ybf}^n)_t|^p \le \underset{\tau\in\mathcal{T}^n_t}{\esssup\,}2^{p/2-1}\E\left[ \int_t^{\tau} \delta^{\frac{p-2}{2}} \eta^{p}C_f^{p}\left( \frac{1}{n}\sum_{j=1}^n e^{p\beta s}|Y_s^{j,n}-\bar{Y}_s^{j,n}|^p \right) ds \right.  \\  \left. \qquad \qquad +\left(\gamma_1 e^{p\beta \tau}|Y^{i,n}_{\tau}-\bar{Y}^{i,n}_{\tau}|+ \gamma_2 \left(\frac{1}{n}\sum_{j=1}^ne^{p\beta \tau}|Y_{\tau}^{j,n}-\bar{Y}_{\tau}^{j,n}|^p \right)^{\frac{1}{p}}\right)^{p}|\mathcal{F}_t\right].
\end{array}
\end{equation} 
Next, let $\delta >0$, $t\in [T-\delta,T]$ and with $\alpha:=\max(2^{p/2-1}\delta^{p/2} \eta^{p}C^{p}_f+2^{3p/2-2}\gamma^p_1, 2^{p/2-1} \delta^{p/2} \eta^{p}C^{p}_f+2^{3p/2-2}\gamma^p_2)$.
 By applying the same arguments used in Step 2 from the proof of Theorem \ref{mf-case}, and the definition \eqref{n-norm-2}, we have
\begin{equation*}\begin{array} {lll}
\|\widetilde{\Phi}(\Ybf^n)-\widetilde{\Phi}(\bar{\Ybf}^n)\|^p_{\lpn_\beta[T-\delta,T]}\le \alpha \|\Ybf^n-\bar{\Ybf}^n\|^p_{\lpn_\beta[T-\delta,T]}.
\end{array}
\end{equation*}
Hence, if $\gamma_1$ and $\gamma_2$ satisfy
$$
\gamma_1^p+\gamma_2^p<2^{2-3p/2},
$$
we can choose 
$$
0<\delta<\left\{\frac{1}{2^{p/2-1}\eta^p C_f^p}\left(1-2^{3p/2-2}(\gamma^p_1+\gamma^p_2)\right) \right\}^{2/p}
$$ 
to make $\widetilde{\Phi}$ a contraction on $\lpn_\beta([T-\delta,T])$, i.e. $\widetilde{\Phi}$ admits a unique fixed point over $[T-\delta,T]$. Moreover, in view of Assumption \ref{generalAssump} (ii)(a), we get that $\Ybf^n\in \spn([T-\delta,T])$.

\medskip
\underline{\textit{Step 3.}} The system \eqref{BSDEParticle} has a unique solution.
We first denote 
\begin{eqnarray*}
&&\bm{\xi}^n := (\xi^{1,n}, \dots, \xi^{n,n});\quad \textbf{Z}^{i,n} := (Z^{i,1,n}, \dots, Z^{i,n,n})_{i=1,\ldots,n}; \\
&&\textbf{U}^{i,n} := (U^{i,1,n}, \dots, U^{i,n,n})_{i=1,\dots,n};\quad \textbf{K}^n := (K^{1,n}, \dots, K^{n,n}). 
\end{eqnarray*}
Given $\textbf{Y}^n = (Y^{1,n},\ldots,Y^{n,n})\in \lpn_\beta$ the solution obtained in \textit{Step 2}, let $(\hat{\textbf{Y}}^n,\hat{\textbf{Z}}^n,\hat{\textbf{U}}^n,\hat{\textbf{K}}^n) \in \mathcal{S}^{p,\otimes n} \otimes \mathcal{H}^{p,n\otimes n} \otimes \mathcal{H}_{\nu}^{p,n\otimes n} \otimes \mathcal{S}_{i}^{p,\otimes n}$ be the unique solution of the RBSDE system over $[T-\delta,T]$, with barrier $(h(t, Y^{i,n}, L_{n}(\textbf{Y}^n)))_{i=1, \ldots, n}$ (which belongs to $\mathcal{S}^{p,\otimes n}$) and driver $(\textbf{f}^{i} \circ \textbf{Y}^n)_{i=1 \dots n}$. Then, for each $i=1,\dots, n$, we get
$$
\hat{Y}^{i,n}_t=\underset{\tau\in\mathcal{T}^n_t}{\esssup\,}\mathcal{E}_{t,\tau}^{\textbf{f}^{i} \circ \textbf{Y}^n}\left[h(\tau,\textbf{Y}^n_{\tau})\ind_{\{\tau<T\}}+ \xi^{i,n} \ind_{\{\tau=T\}}\right].
$$
By the fixed point argument from \textit{Step 2}, we have $\hat{\textbf{Y}}^n=\textbf{Y}^n$ a.s., where the equality is component-wise. This implies that $\textbf{Y}^n \in \mathcal{S}^{p,\otimes n}$ and for all $1 \leq i \leq n$ and $T-\delta \leq t \leq T$,
\begin{align}
Y^{i,n}_t &= \xi^{i,n} +\int_t^T f(s,Y^{i,n}_{s},  \hat{Z}^{i,i,n}_s,\hat{U}^{i,i,n}_s, L_{n}[\textbf{Y}^n_{s}])ds + \hat{K}^{i,n}_T - \hat{K}^{i,n}_t -\int_t^T \sum_{j=1}^n \hat{Z}^{i,j,n}_s dB^j_s \nonumber \\ &-\int_t^T \int_{R^*} \sum_{j=1}^n \hat{U}^{i,j,n}_s(e) \tilde N^j(ds,de).
\end{align}

Therefore, we obtain existence of a solution.
Note that, for each $i$, the process $Y^{i}_t$ is a \textit{nonlinear} $\textbf{f}^{i} \circ \textbf{Y}^n$-supermartingale. Hence, by the uniqueness of the  \textit{nonlinear} Doob-Meyer decomposition, we get uniqueness of the associated processes $(\hat{\bf{Z}}^{i,n},\hat{\bf{U}}^{i,n},\hat{K}^{i,n})$. This yields existence and uniqueness of \eqref{BSDEParticle} on the time interval $[T-\delta, T]$ in this general case.

Applying the same reasoning on each time interval $[T-(j+1)\delta,T-j\delta]$, $1 \leq j \leq m$, with a similar dynamics but terminal condition $\textbf{Y}^{i}_{T-j\delta}$ at time $T-j\delta$, we build recursively, for $j=1$ to $n$, a solution $(\textbf{Y}^i,\textbf{Z}^i,\textbf{U}^i,\textbf{K}^i)$ on each time interval $[T-(j+1)\delta,T-j\delta].$ Pasting properly these processes, we naturally derive a unique solution $(\textbf{Y},\textbf{Z},\textbf{U},\textbf{K})$ satisfying (\ref{BSDEParticle}) on the full time interval $[0,T]$. \qed
\end{proof}


\section{Convergence of the particle system and propagation of chaos}\label{convergence}
\noindent This section is concerned with the convergence of the sequence of processes $(Y^{i,n},Z^{i,n},U^{i,n},K^{i,n})$, solutions of the particle system \eqref{BSDEParticle}, to the independent copies $(Y^{i},Z^{i},U^{i},K^{i})$ of the solution of the mean-field Reflected BSDEs \eqref{BSDE1}, as well as a related propagation of chaos result.

\subsection{Preliminaries}
\medskip
We first introduce the Wasserstein metric on $\mathbb{D}:=D([0,T],\R)$ and recall some of its properties along with a law of large numbers result for $\mathbb{D}$-valued i.i.d. random variables to be used in our proofs towards the propagation of chaos property.\\

\noindent \textbf{Wasserstein metric on
$\mathbb{D}:=D([0,T],\R)$}\\

\noindent Let $\mathbb{D}=D([0,T],\R)$ be the set of c{\`ad}l{\`a}g functions i.e. the space of functions from $[0, T]$ to $\R$ which are right continuous with left limits at each $t\in [0,T)$ and are left continuous at time $T$. Let $\Lambda$ denote the class of strictly increasing, continuous mappings of $[0,T]$ onto itself. If $\lambda \in \Lambda$, then $\lambda(0) = 0$ and $\lambda(T) = T$. Let $d^o(x,y)$ be the infimum of those positive $\epsilon$ for which $\Lambda$ contains some $\lambda$ such that $||\lambda||^o:=\underset{0\le s<t\le T}{\sup}|\log \frac{\lambda (t)-\lambda (s)}{t-s}|\leq \epsilon$ and for $x,y \in \mathbb{D}$,
$$\underset{0\le t\le T}{\sup}|x(t)-y(\lambda (t))|=\underset{0\le t\le T}{\sup}|x(\lambda^{-1}(t))-y(t)|<\epsilon. $$ In other words, let
\begin{align}\label{d-o}
    d^o(x,y) = \underset{\lambda \in \Lambda}{\inf}\{||\lambda||^o\vee \underset{0\le t\le T}{\sup}|x(t)-y(\lambda(t))|\}.
\end{align}
The space $\mathbb{D}$ endowed with the Skorohod metric $d^o$ is a Polish space. 
By taking $\lambda$ to be identity function, we obtain $\|\lambda\|^0=0$, and then
$$
d^o(x,y)\le |x-y|_T,
$$
where $|z|_t:=\underset{0\le s\le t}{\sup}|z(s)|$, i.e.  $|\cdot|_T$ induces a stronger topology than $d^0$. 

\noindent To derive \eqref{x-y-esti} below, we need the metric $d$ which is equivalent to $d^o$  and for which $(\mathbb{D},d)$ is  separable but not complete, defined as (see (14.12)-(14.13) in Billingsley \cite{b68})
\begin{align}\label{d}
    d(x,y) = \underset{\lambda \in \Lambda}{\inf} 
    \{ \underset{0\le t\le T}\sup \,|\lambda(t)-t|\vee \underset{0\le t\le T}\sup\,|x(t)-y(\lambda(t))|\}.
\end{align}
In view of \eqref{d-o} and \eqref{d}, we have, for any $(x,y)\in \mathbb{D}\times\mathbb{D}$,
\begin{equation}\label{d-d-o}
  d(x,y)\le |x-y|_T,\quad  d(x,y)\le T[e^{d^o(x,y)}-1],
\end{equation}
where the last inequality follows from the fact that, by using the continuity of the map $\lambda$, we get
$$\begin{array}{lll}
\underset{0\le t\le T}\sup \,|\lambda(t)-t|=\underset{0 < t\le T}\sup \,|\lambda(t)-t|=\underset{0 < t\le T}\sup \,t|\frac{\lambda(t)-0}{t-0}-1| \\ \qquad\qquad\qquad\quad\le T\underset{0 < t\le T}\sup \,|\frac{\lambda(t)-0}{t-0}-1|\le T[e^{\|\lambda\|^o}-1],
\end{array}
$$
where we have used the inequality $|a-1|\le e^{|\log{a}|}-1$ for $a>0$. 

Following Billingsley \cite{b68}, (14.6) and (14.7), for $x\in \mathbb{D}$, $\delta >0$ and $0\le u\le v\le T$, put
$$\begin{array}{ll}
w_x([u,v])=\sup\{|x(t)-x(s)|:\, s,t \in[u,v]\}, \\
w^{\prime}_x(\delta)=\underset{\{t_i\}}{\inf}\,\underset{1\le i\le r}{\max} \,w_x([t_{i-1},t_i)),
\end{array}
$$
where the infimum  extends over the finite sets $\{t_i\}$ of points satisfying 
$$
\left\{ \begin{array}{lll}
0=t_0<t_1<\cdots<t_r=T, \\
t_i-t_{i-1}>\delta,\quad i=1,\ldots, r.
\end{array}
\right.
$$
From this definition, it follows that
\begin{equation}\label{w-w}
w_x([u,v))\le 2 w^{\prime}_x(v-u)+\underset{u< t\le v}{\sup}\,|x(t)-x(t^-)|,\quad \text{for all}\,\, 0\le u\le v\le T.
\end{equation}
Furthermore, for any $x,y \in\mathbb{D}$ and any $\lambda\in \Lambda$, it holds that
$$
|x(t)-y(t)|\le |x(t)-y(\lambda(t))|+|y(\lambda(t))-y(t)|.
$$
The definition \eqref{d-o}, which implies that for any $\delta >0$, for any $x$, $y \in \mathbb{D}$ such that $d^o(x,y)\le \ln(1+ \delta/T)$ we get  $d(x,y)\le \delta$ by \eqref{d-d-o}, which together with the above relation, yields
$$
|x(t)-y(t)|\le \delta +w_y([t-\delta,t+\delta)).
$$
Thus, in view of \eqref{w-w}, if $d^o(x,y)\le \ln(1+ \delta/T)$ then
\begin{equation}\label{x-y-esti}
|x(t)-y(t)|\le \delta+ 2 w^{\prime}_y(2\delta)+\underset{t-\delta< s\le t+\delta}{\sup}\,|y(s)-y(s^-)|.
\end{equation}

The space of probability measures on the space $\mathbb{D}$ is denoted by $\mathcal{P}(\mathbb{D})$. We define the $p$-Wasserstein metric between two probability measures $P$ and $Q$ on $(\mathbb{D},|\cdot|_T)$ as
\begin{equation}\label{W-filt}
D_t(P,Q):=\inf\left\{\left(\int_{\mathbb{D}\times \mathbb{D}}|x-y|_t^p R(dx,dy)\right)^{1/p}\right\},\quad 0\le t\le T,
\end{equation}
over $R\in\mathcal{P}(\mathbb{D}\times \mathbb{D})$ with marginals $P$ and $Q$. 

\noindent Let $\mathcal{P}_p(\mathbb{D})$ denote the space of probability measures $Q$ on $(\mathbb{D},|\cdot|_T)$ with finite $p$th-moments: $$
\|Q\|_p^p:=\int_{\mathbb{D}}|w|^p_TQ(dw)<+\infty.
$$ 
For $P, Q\in \mathcal{P}_p(\mathbb{D})$, we have
\begin{equation}\label{ordering}
D_s(P,Q)\le D_t(P,Q),\quad 0\le s\le t.
\end{equation}
Moreover,  for $P, Q\in \mathcal{P}_p(\mathbb{D})$ with time marginals $P_t:=P\circ \omega^{-1}(t)$ and $Q_t:=Q\circ \omega^{-1}(t)$ ($\omega$ being the coordinate process in $\mathbb{D}$), the $p$-Wasserstein  distance between $P_t$ and $Q_t$ satisfies
\begin{equation}\label{margine-1}
\mathcal{W}_p(P_t,Q_t) \le D_t(P,Q),\quad 0\le t\le T.
\end{equation}
In particular,
\begin{equation}\label{margine-2}
\underset{0\le t\le T}{\sup}\mathcal{W}_p(P_t,Q_t)\le D_T(P,Q).
\end{equation}
Endowed with the $p$-Wasserstein metric $D_T$,  $\mathcal{P}_p(\mathbb{D})$ is a complete metric space.

The $p$-Wasserstein metric for probability measures on $(\mathbb{D},d^o)$ is defined by
$$
D_T^o(P,Q):=\inf\left\{\left(\int_{\mathbb{D}\times \mathbb{D}}d^o(x,y)^p R(dx,dy)\right)^{1/p}\right\}
$$
over $R\in\mathcal{P}(\mathbb{D}\times \mathbb{D})$ with marginals $P$ and $Q$. Endowed with the $p$-Wasserstein metric $D^o_T$, $\mathcal{P}(\mathbb{D})$ is a Polish space.

\noindent Due to the relation $d^o(x,y)\le |x-y|_T$,
we have
$$
D_T^o(P,Q)\le D_T(P,Q),\quad P,Q\in \mathcal{P}_p(\mathbb{D}),
$$
which means that the metric $D_T$ is stronger than $D_T^o$. 

\medskip

We end these preliminaries with the following law of large numbers.

\medskip \noindent
\begin{Lemma}\label{LLN}
Let $P\in\mathcal{P}_p((\mathbb{D},d^o))$ and $L^n(\zeta):=\frac{1}{n}\sum_{i=1}^n \delta_{\zeta^i}$ denote the empirical measure associated with the sequence of i.i.d. $\mathbb{D}$-valued random variables $\zeta:=(\zeta^i)_{i\ge 1}$ with probability law $P$. Then, 
\begin{equation}\label{GC-1}
\text{for all}\,\,\, n\ge 1,\quad \E[D_T^o(L^n(\zeta),P)^p]\le 2^{p} D_T^o(P,\delta_0)^p\quad\text{and} \quad \lim_{n\to\infty}\E[D_T^o(L^n(\zeta),P)^p]=0.
\end{equation}
\end{Lemma}

\begin{proof}
The random variables $x_i:=\delta_{\zeta_i}$ being i.i.d. with values in the Polish space $(\mathcal{P}(\mathbb{D}),D^o_T)$, by the strong law of large numbers (SLLN) (Varadarajan's theorem), with probability one, 
$L^n(\zeta):=\frac{1}{n}\sum_{i=1}^n x_i$ converges weakly to $P$, as $n$ tends to infinity. Moreover, since $P\in\mathcal{P}_p((\mathbb{D},d^0))$, by the SLLN, $\frac{1}{n}\sum_{i=1}^n d^0(\zeta_i,0)^p$ converges to $D_T^o(P,\delta_0)^p$ a.s.. Therefore, in view of Corollary 12.2.2 in \cite{Rachev},  $D_T^o(L^n(\zeta),P)$ converges to $0$ a.s., as $n\to \infty$.
We also have
$$
D_T^o(L^n(\zeta),P)^p\le 2^{p-1}(\frac{1}{n}\sum_{i=1}^nd^0(\zeta_i,0)^p+ D_T^o(P,\delta_0)^p).
$$
Thus,
$$
\E[D_T^o(L^n(\zeta),P)^p]\le 2^{p-1}( \E[\frac{1}{n}\sum_{i=1}^n d^0(\zeta_i,0)^p]+ D_T^o(P,\delta_0)^p)=2^{p}D_T^o(P,\delta_0)^p.
$$
Therefore, $D_T^o(L^n(\zeta),P)^p$ is uniformly integrable. Hence it converges in $L^1$, i.e.
$$
\lim_{n\to\infty}\E[D_T^o(L^n(\zeta),P)^p]=0. $$\qed
\end{proof}

 The propagation of chaos is a property of systems of exchangeable (symmetric in law) processes. We first recall the definition of exchangeable (or symmetric in law) random variables and then of that of propagation of chaos. We refer to Sznitman \cite{s91} for further details.

\begin{Definition}[Exchangeable r.v.] The random variables $X^1,X^2,\ldots X^n$ are said to be exchangeable if the law of the random vector $(X^1,X^2,\ldots,X^n)$ is the same as that of $(X^{\sigma(1)},X^{\sigma(2)},\ldots,X^{\sigma(n)})$ for every  permutation $\sigma$ of the set $\{1,2,\ldots,n\}$. We write
$$
\text{Law}\,(X^1,X^2,\ldots,X^n)=\text{Law}\,(X^{\sigma(1)},X^{\sigma(2)},\ldots,X^{\sigma(n)}).
$$
\end{Definition}

Let $E$ be a separable metric space and let $P$ be a probability measure on $E$.
\begin{Definition}[Propagation of chaos property]
The sequence of exchangeable $E$-valued random variables $X^{1,n},X^{2,n},\ldots, X^{n,n}, \,n\ge 1$, is said to be $P$-chaotic provided that, for any fixed positive integer $k$,
$$
\underset{n\to\infty}{\lim}\,\text{Law}\,(X^{1,n},X^{2,n},\ldots, X^{k,n})=P^{\otimes k}.
$$
\end{Definition}

\subsection{Propagation of chaos}
To establish the propagation of chaos property of the particle system \eqref{BSDEParticle}, we shall further impose the following conditions.

\begin{Assumption}\label{Assump:chaos}
\begin{itemize}
    \item [(i)] The sequence  $ \xi^n=(\xi^{1,n},\xi^{2,n},\ldots,\xi^{n.n})$ is exchangeable i.e. the sequence of probability laws $\mu^n$ of $\xi^n$ on $\R^n$ is symmetric.

    \item[(ii)] For each $i\ge 1$, $\xi^{i,n}$ converges in $L^p$ to $ \xi^i$, i.e. $$\lim_{n\to\infty}\E[|\xi^{i,n}-\xi^i|^p]=0,$$
        
        where the random variables $\xi^{i}\in L^p(\Fc^i_T)$ are independent and equally distributed (iid) with probability law $\mu$.
    \item[(iii)] The component $(Y_t)$ of the unique solution of the mean-field reflected BSDE  \eqref{BSDE1} has jumps only at totally inaccessible stopping times.
\end{itemize}
    \end{Assumption}
Sufficient conditions on the barrier $h$ under which $(iv)$ holds are given in Theorem \ref{inacces}.\\

For $m\ge 1$, introduce the Polish spaces
$\mathbb{H}^{2,m}:=L^2([0,T];\R^m)$ and $\mathbb{H}_{\nu}^{2,m}$, the space of measurable functions $\ell: \,[0,T]\times \R^*\longrightarrow \R^n$ such that $\int_0^T\int_{\R^*}\sum_{j=1}^m |\ell^j(t,u)|^2\nu(du)dt< \infty$.     
    
\medskip
In the following proposition, we show that the \textit{exchangeability property} transfers from the terminal conditions to the the associated solution processes.
    \begin{Proposition}[Exchangeability property]\label{exchange}
Assume the sequence  $ \xi^n=(\xi^{1,n},\xi^{2,n},\ldots,\xi^{n.n})$ is exchangeable i.e. the sequence of probability laws $\mu^n$ of $\xi^n$ on $\R^n$ is symmetric. Then the processes $\Theta^{i.n}:=(Y^{i,n},Z^{i,n},U^{i,n},K^{i,n}),\,i=1,\ldots, n,$ solutions of  the systems \eqref{BSDEParticle} are exchangeable.
\end{Proposition}

\begin{proof} Since 
for any permutation $\sigma$ of the set $\{ 1,2,\ldots,n\}$, we have
$$
\frac{1}{n}\sum_{i=1}^n\delta_{Y^{i,n}}=\frac{1}{n}\sum_{i=1}^n\delta_{Y^{\sigma(i),n}},
$$
by the pathwise existence and uniqueness results obtained in Theorem \ref{existParticle-2}, we have
$$
(\Theta^{1.n},\Theta^{2.n},\ldots,\Theta^{n.n})=(\Theta^{\sigma(1).n},\Theta^{\sigma(2).n},\ldots,\Theta^{\sigma(n).n}),\quad \text{a.s.}
$$
whenever $(\xi^{1,n},\ldots,\xi^{n,n})=(\xi^{\sigma(1),n},\ldots,\xi^{\sigma(n),n})$ a.s.. Now, $\xi^n$ being exchangeable, by a careful use of the Skorohod representation theorem applied to the Polish space  $$
\R^n\times (\mathbb{D},d^0)^{\otimes n}\times C([0,T];\R^n)\times(\mathbb{D},d^0)^{\otimes n}\times \mathbb{H}^{2,n\otimes n}\times \mathbb{H}_{\nu}^{2,n\otimes n}\times (\mathbb{D},d^0)^{\otimes n}
$$
where the process $(\xi^n,\{B^i\}_{i=1}^n, \{\tilde{N}^i\}_{i=1}^n,\{\Theta^{i.n}\}_{i=1}^n)$ take their values, as used by Yamada and Watanabe \cite{yw71} to
prove uniqueness of solutions to SDEs and extended  in 
Delarue \cite{Delarue} (see the proof after Remark 1.6, pp. 224-227) to prove uniqueness in law for forward-backward SDEs, it follows that each of the systems is exchangeable i.e. for every permutation $\sigma$ of $\{1,2,\ldots,n\}$,
$$
\text{Law}(\Theta^{1,n},\Theta^{2,n},\ldots,\Theta^{n,n})=\text{Law}(\Theta^{\sigma(1),n},\Theta^{\sigma(2),n},\ldots,\Theta^{\sigma(n),n}).$$
\qed
\end{proof}

Consider the product space 
$$
G:=\mathbb{D}\times \mathbb{H}^{2,n}\times \mathbb{H}_{\nu}^{2,n}\times \mathbb{D}
$$
endowed with the product metric 
$$
\delta((y,z,u,k),(y^{\prime},z^{\prime},u^{\prime},k^{\prime})):=\left(d^o(y,y^{\prime})^p+\|z-z^{\prime}\|^p_{\mathbb{H}^{2,n}}+\|u-u^{\prime}\|^p_{\mathbb{H}_{\nu}^{2,n}}+d^o(k,k^{\prime})^p\right)^{\frac{1}{p}}.
$$
We define the Wasserstein metric on $\mathcal{P}_p(G)$ by 
\begin{equation}\label{W-simultan} 
D_G(P,Q)=\inf\left\{\left(\int_{G\times G}  \delta((y,z,u,k),(y^{\prime},z^{\prime},u^{\prime},k^{\prime}))^p R(d(y,z,u,k),d(y^{\prime},z^{\prime},u^{\prime},k^{\prime}))\right)^{1/p}\right\},
\end{equation}
over $R\in\mathcal{P}(G\times G)$ with marginals $P$ and $Q$. Since $(G,\delta)$ is a Polish space, $(\mathcal{P}_p(G), D_G)$ is a Polish space and induces the topology of weak convergence (see Theorem 6.9 in \cite{Villani}).

\medskip
Let $\Theta^i:=(Y^i,Z^i,U^i,K^i), i=1,\ldots, n$, with independent terminal values $Y^i_T=\xi^i,i=1,\ldots, n$, be independent copies of $\Theta:=(Y,Z,U,K)$, the solution of \eqref{BSDE1}. 
More precisely, for each $i=1,\ldots,n$, $\Theta^i$ is the unique solution of the reflected MF-BSDE
\begin{align}\label{BSDEParticle-Theta}
    \begin{cases}
    Y^{i}_t = \xi^{i} +\int_t^T f(s,Y^{i}_{s},Z^{i}_{s},U^{i}_{s},\mathbb{P}_{Y_s^{i}})ds + K^{i}_T - K^{i}_t- \int_t^T Z^{i}_s dB^i_s -  \int_t^T \int_{R^*}  U^{i}_s(e) \tilde N^i(ds,de), \\
    Y^{i}_{t} \geq h(t,Y^{i}_{t},\mathbb{P}_{Y_t^{i}}),  \quad \forall t \in [0,T], \\
    \int_0^T (Y^{i}_{t^-}-h(t^-,Y^{i}_{t^-},\mathbb{P}_{Y_{t^-}^{i}}))dK^{i}_t = 0.
    \end{cases}
\end{align}
In the sequel, we denote $(f \circ Y^{i})(t,y,z):=f(t,y,z,\mathbb{P}_{Y^{i}_t})$.

\noindent For any fixed $1\leq k\leq n$,  let
$$
\P^{k,n}:=\text{Law}\,(\Theta^{1,n},\Theta^{2,n},\ldots,\Theta^{k,n}),\quad \P_{\Theta}^{\otimes k}:=\text{Law}\,(\Theta^1,\Theta^2,\ldots,\Theta^k).
$$
be the joint probability laws of the processes $(\Theta^{1,n},\Theta^{2,n},\ldots,\Theta^{k,n})$ and $(\Theta^1,\Theta^2,\ldots,\Theta^k)$, respectively.

From the definition of the distance $D_G$, we obtain the following inequality. 
\begin{equation}\label{chaos-0}
D_G^p(\P^{k,n},\P_{\Theta}^{\otimes k})\le 
k\sup_{i\le k}\left(\|Y^{i,n}-Y^i\|^p_{\sp}+\|Z^{i,n}-Z^i{\bf e}_i\|^p_{\mathcal{H}^{p,n}}+\|U^{i,n}-U^i{\bf e}_i\|^p_{\mathcal{H}^{p,n}_{\nu}}+\|K^{i,n}-K^i\|^p_{\sp}\right),
\end{equation}
where for each $i=1,\ldots, n$,  ${\bf e}_i:=(0,\ldots,0,\underbrace{1}_{i},0,\ldots,0)$.

Before we state and prove the propagation of chaos result, we derive the following convergence result of the empirical laws of i.i.d. copies $Y^1,Y^2,\ldots,Y^n$ solutions of \eqref{BSDEParticle-Theta} and convergence of the particle system \eqref{BSDEParticle} to the solution to \eqref{BSDE1}.

\begin{Theorem}[Law of Large Numbers]\label{LLN-2} Let $Y^1,Y^2,\ldots,Y^n$ with terminal values $Y^i_T=\xi^i$ be independent copies of  the solution $Y$ of \eqref{BSDE1}. Then, we have
\begin{equation}\label{GC-2}
\lim_{n\to\infty}\E\left[\underset{0\le t\le T}\sup\, \mathcal{W}_p^p(L_n[\textbf{Y}_t],\P_{Y_t})\right]=0.
\end{equation}
\end{Theorem}

\begin{proof}  
For any joint law $R^n( dx,dy)$ with marginals $L_n[\textbf{Y}]$ and $\P_{Y}$ and any $\delta>0$, we have
$$
\begin{array}{lll} \mathcal{W}_p^p(L_n[\textbf{Y}_t],\P_{Y_t})\le \int_{\mathbb{D}\times\mathbb{D}} |x(t)-y(t)|^p R^n( dx,dy) \\ \qquad\qquad\qquad\qquad = \int_{\mathbb{D}\times\mathbb{D}} |x(t)-y(t)|^p\ind_{\{d^o(x,y)\le \ln(1+\delta/T)\}}R^n( dx,dy) \\ \qquad\qquad\qquad\qquad\qquad +
\int_{\mathbb{D}\times\mathbb{D}} |x(t)-y(t)|^p \ind_{\{d^o(x,y)> \ln(1+\delta/T)\}}R^n( dx,dy).
\end{array}
$$
Integrating both sides of \eqref{x-y-esti} against the joint law $R^n( dx,dy)$, we obtain
$$
\begin{array}{lll}
3^{1-p}\int_{\mathbb{D}\times\mathbb{D}} |x(t)-y(t)|^p \ind_{\{d^o(x,y)\le \ln(1+\delta/T)\}}R^n( dx,dy) \\ \qquad\qquad \le \delta^p +2^p\int_{\mathbb{D}} w^{\prime}_y(2\delta)^p\P_Y(dy)+\int_{\mathbb{D}} \underset{(t-\delta) \vee 0< s\le t+\delta}{\sup}\,|y(s)-y(s^-)|^p\P_Y(dy).
\end{array}
$$
Noting that 
$$
\int_{\mathbb{D}} \underset{(t-\delta) \vee 0< s\le t+\delta}{\sup}\,|y(s)-y(s^-)|^p\P_Y(dy)=\E\left[\underset{(t-\delta) \vee 0< s\le t+\delta}{\sup}\,|\Delta Y(s)|^p\right]
$$
where $\Delta Y(s):=Y(s)-Y(s^-)$ denotes the jump size of the process $Y$ solution to \eqref{BSDE1}, we have
$$
\begin{array}{lll}
\mathcal{W}_p^p(L_n[\textbf{Y}_t],\P_{Y_t})\le \int_{\mathbb{D}\times\mathbb{D}} |x(t)-y(t)|^p\ind_{\{d^o(x,y)>\ln(1+\delta/T)\}}R^n( dx,dy)+ 3^{p-1}\delta^p \\ \qquad\qquad\qquad\qquad + 6^p\int_{\mathbb{D}} w^{\prime}_y(2\delta)^p\P_Y(dy)+3^{p-1}\E\left[\underset{(t-\delta) \vee 0< s\le t+\delta}{\sup}\,|\Delta Y(s)|^p\right].
\end{array}
$$
Since, for any $M>0$, it holds that
$$
\max(|x-y|_T^p-M^p,0)\le 2^p|x|_T^p\ind_{\{|x|_T\ge M/2\}}+2^p|y|_T^p\ind_{\{|y|_T\ge M/2\}},
$$
we have
$$\begin{array}{lll}
\int_{\mathbb{D}\times\mathbb{D}} |x-y|^p_T\ind_{\{d^o(x,y)>\ln(1+\delta/T)\}}R^n( dx,dy)
\\ \qquad\qquad \le \int_{\mathbb{D}\times\mathbb{D}} \min(|x-y|_T,M)^p\ind_{\{d^o(x,y)>\ln(1+\delta/T)\}}R^n( dx,dy) \\ \qquad\qquad + 2^p\int_{\mathbb{D}\times\mathbb{D}}|x|_T^p\ind_{\{|x|_T\ge M/2,d^o(x,y)>\ln(1+\delta/T) \}}R^n( dx,dy) \\ \qquad\qquad  +2^p \int_{\mathbb{D}\times\mathbb{D}}|y|_T^p\ind_{\{|y|_T\ge M/2, d^o(x,y)>\ln(1+\delta/T)\}}R^n( dx,dy) \\ \qquad\qquad
\le  M^p\int_{\mathbb{D}\times\mathbb{D}} \ind_{\{d^o(x,y)>\ln(1+\delta/T)\}}R^n( dx,dy)+  2^p\int_{\mathbb{D}}|x|_T^p\ind_{\{|x|_T\ge M/2\}}L_n[\mathbf{Y}](dx) \\ \qquad\qquad 
+ 2^p\int_{\mathbb{D}}|y|_T^p\ind_{\{|y|_T\ge M/2\}}\P_{Y}(dy).
\end{array}
$$
Therefore, for any $M>0, \delta>0$,
\begin{equation}\label{W-est-1}
\begin{array}{lll}
\underset{0\le t\le T}\sup\,\mathcal{W}_p^p(L_n[\textbf{Y}_t],\P_{Y_t})\le  M^p\int_{\mathbb{D}\times\mathbb{D}} \ind_{\{d^o(x,y)>\ln(1+\delta/T)\}}R^n( dx,dy) \\ \qquad\qquad +  2^p\int_{\mathbb{D}}|x|_T^p\ind_{\{|x|_T\ge M/2\}}L_n[\mathbf{Y}](dx) 
+ 2^p\int_{\mathbb{D}}|y|_T^p\ind_{\{|y|_T\ge M/2\}}\P_{Y}(dy) \\ \qquad\qquad + 3^{p-1}\delta^p+6^p\int_{\mathbb{D}} w^{\prime}_y(2\delta)^p\P_Y(dy)+3^{p-1}\underset{0\le t\le T}{\sup}\,\E\left[\underset{(t-\delta) \vee 0< s\le t+\delta}{\sup}\,|\Delta Y(s)|^p\right].
\end{array}    
\end{equation}
By (14.8) in Billingsley \cite{b68}, $\underset{\delta\to 0}\lim\, w^{\prime}_y(2\delta)=0$, and since 
$$
\int_{\mathbb{D}} w^{\prime}_y(2\delta)^p\P_Y(dy)\le 2\E[|Y|^p_T] <\infty,
$$
by Lebesgue's dominated convergence theorem, we obtain
$$
\underset{\delta\to \infty}{\lim}\,\int_{\mathbb{D}} w^{\prime}_y(2\delta)^p\P_Y(dy)=0.
$$
Moreover, for any $\delta >0$,
$$
\underset{n\to \infty}{\lim}\,\int_{\mathbb{D}\times\mathbb{D}} |x-y|_T^p\ind_{\{d^o(x,y)>\ln(1+\delta/T)\}}R^n( dx,dy)= 0, \quad\text{a.s.}
$$
Indeed, in view of \eqref{GC-1}, for any $\delta >0$, we have
$$
\underset{n\to \infty}{\lim}\,\int_{\mathbb{D}\times\mathbb{D}} \ind_{\{d^o(x,y)>\ln(1+\delta/T)\}}R^n( dx,dy)\le  (\ln(1+\delta/T))^{-1}\underset{n\to \infty}{\lim}\,D_T^o(L_n[\mathbf{Y}],\P_{Y})^p=0,\quad\text{a.s.}
$$
Moreover, by the SLLN,
$$
\underset{n\to \infty}{\lim}\, \int_{\mathbb{D}}|x|_T^p\ind_{\{|x|_T\ge M/2\}}L_n[\mathbf{Y}](dx)=\int_{\mathbb{D}}|y|_T^p\ind_{\{|y|_T\ge M/2\}}\P_{Y}(dy)
$$
and since $\P_{Y}\in \mathcal{P}_p(\mathbb{D})$, it is uniformly integrable. In particular, we have
$$
\underset{M\to \infty}{\lim}\,\int_{\mathbb{D}}|y|_T^p\ind_{\{|y|_T\ge M/2\}}\P_{Y}(dy)=0.
$$
Therefore, for any $\delta >0$
$$
\begin{array}{lll}
\underset{n\to \infty}{\lim}\,\int_{\mathbb{D}\times\mathbb{D}} |x-y|_T\ind_{\{d^o(x,y)>\ln(1+\delta/T)\}}R^n( dx,dy)\le 2^{p+1}\underset{M\to \infty}{\lim}\,\int_{\mathbb{D}}|y|_T^p\ind_{\{|y|_T\ge M/2\}}\P_{Y}(dy)=0 \quad \text{a.s.}
\end{array} 
$$
Finally, by taking the limits $n\to\infty$, then $M\to\infty$ and lastly $\delta\to 0$, we obtain
\begin{equation}\label{W-est-2}
\underset{n\to \infty}{\lim}\,
\underset{0\le t\le T}{\sup}\,\mathcal{W}_p^p(L_n[\textbf{Y}_t],\P_{Y_t})\le 3^{p-1}
\underset{\delta\to 0}\lim\,\underset{0\le t\le T}{\sup}\E[\underset{(t-\delta) \vee 0< s\le t+\delta}{\sup}\,|\Delta Y(s)|^p],\quad\text{a.s.}
\end{equation}
Once we can show that
\begin{equation}\label{jumps}
\underset{\delta\to 0}\lim\,\underset{0\le t\le T}{\sup}\E\left[\underset{(t-\delta) \vee 0< s\le t+\delta}{\sup}\,|\Delta Y(s)|^p\right]=0,
\end{equation}
we obtain
$$
\underset{n\to \infty}{\lim}\,
\underset{0\le t\le T}{\sup}\,\mathcal{W}_p^p(L_n[\textbf{Y}_t],\P_{Y_t})=0 \quad \text{a.s.}
$$
But, since 
$$
\E\left[\underset{0\le t\le T}{\sup}\,\mathcal{W}_p^p(L_n[\textbf{Y}_t],\P_{Y_t})\right]\le \E[D_T^p(L_n[\textbf{Y}],\P_{Y})]\le 2^{p-1}\left( \E\left[\frac{1}{n}\sum_{i=1}^n|Y^i|_T^p\right]+ \|\P_{Y}\|_p^p\right)=2^{p}\|\P_{Y}\|_p^p, 
$$
the random variable $\underset{0\le t\le T}{\sup}\,\mathcal{W}_p^p(L_n[\textbf{Y}_t],\P_{Y_t})$ is uniformly integrable. Thus, it converges to $0$ in $L^1$, i.e.
$$
\underset{n\to \infty}{\lim}\,\E\left[\underset{0\le t\le T}{\sup}\,\mathcal{W}_p^p(L_n[\textbf{Y}_t],\P_{Y_t})\right]=0.
$$
It remains to show \eqref{jumps}. By Assumption $\eqref{Assump:chaos}(iv)$, it follows that the process $K_t$ is continuous, which implies that $\Delta Y_t=\Delta J_t$, where  $J_t:=\int_{(0,t]}\int_{\R^*}U_s(e)N(ds,de)$.

\medskip

$\bullet$ \underline{Estimation of $\Delta J_s$ (for the case $p=2$)}.
Clearly, 
$$
\Delta J_t=U_t(\Delta\eta_t)\ind_{\R^*}(\Delta\eta_t),
$$
where $(\eta_s)_s$ denotes the L{\'e}vy  process for which
$$
N(t,A):=\#\{s\in(0,t],\,\, \Delta\eta_t\in A\},\quad A\subset \R^*.
$$
We have
$$\begin{array}{lll}
\underset{(t-\delta) \vee 0<s\le t+\delta}{\sup}\,|\Delta J_s|^2=\underset{(t-\delta) \vee 0<s\le t+\delta}{\sup}\,
|U_s(\Delta\eta_s)|^2\ind_{\R^*}(\Delta\eta_s)\le 
\underset{(t-\delta) \vee 0<s\le t+\delta}{\sum}\,
|U_s(\Delta\eta_s)|^2\ind_{\R^*}(\Delta\eta_s)
\end{array}
$$
that is
$$
\underset{(t-\delta) \vee 0 <t \le t+\delta}{\sup}\,|\Delta J_s|^2\le \int_{(0,T]\times \R^*}\ind_{((t-\delta) \vee 0,t+\delta]}(s)|U_s(e)|^2N(ds,de). $$
Since $U\in \mathcal{H}_{\nu}^2$,  we have
$$
\E\left[\int_{(0,T]\times \R^*}\ind_{((t-\delta) \vee 0,t+\delta]}(s)|U_s(e)|^2N(ds,de)\right]= \E\left[\int_{(t-\delta) \vee 0}^{t+\delta}|U_s|^2_{\nu}ds\right].
$$
Therefore, for any $R>0$, 
$$\begin{array}{lll}
\E\left[\underset{(t-\delta) \vee 0<t\le t+\delta}{\sup}\,|\Delta J_s|^2\right]  \le R^2 \delta+\E\left[\int_0^T|U_s|_{\nu}^2\ind_{\{|U_s|_{\nu}>R\}}ds\right],
\end{array}
$$
so
\begin{equation}\label{J-2}
\underset{0\le t\le T}{\sup}\E\left[\underset{(t-\delta) \vee 0<t\le t+\delta}{\sup}\,|\Delta J_s|^2\right]\le R^2 \delta+\E\left[\int_0^T|U_s|_{\nu}^2\ind_{\{|U_s|_{\nu}>R\}}ds\right].
\end{equation}
Thus,
$$
\underset{\delta\to 0}\lim\,\underset{0\le t\le T}{\sup}\E\left[\underset{(t-\delta) \vee 0<t\le t+\delta}{\sup}\,|\Delta J_s|^2\right]\le
\E\left[\int_0^T|U_s|_{\nu}^2\ind_{\{|U_s|_{\nu}>R\}}ds\right].
$$
Again, since $U\in \mathcal{H}_{\nu}^2$, $\underset{R \to \infty}\lim\,\E\left[\int_0^T|U_s|_{\nu}^2\ind_{\{|U_s|_{\nu}>R\}}ds\right]=0$. Hence,
$$\underset{\delta\to 0}\lim\,\underset{0\le t\le T}{\sup}\E\left[\underset{(t-\delta) \vee 0<t\le t+\delta}{\sup}\,|\Delta J_s|^2\right]=0.
$$

\medskip
$\bullet$ \underline{Estimation of $\Delta J_s$ (for the case $p>2$)} 

Since the process $(K_t)$ is continuous, we obtain, for any $L>0$,
\begin{align*}
 \underset{(t-\delta) \vee 0<s\le t+\delta}{\sup}\,|\Delta J_s|^p &  \leq 2^{p-2} \underset{(t-\delta) \vee 0<s\le t+\delta}{\sup}\,
|\Delta J_s|^2\underset{0 \le s\le T}{\sup}\,|Y_s|^{p-2} \ind_{\{\underset{0 \le s\le T}{\sup}\,|Y_s|\le L \}}+2^{p}\underset{0 \le s\le T}{\sup}\,|Y_s|^{p} \ind_{\{\underset{0 \le s\le T}{\sup}\,|Y_s|>L\}} \\ & \le
2^{p-2}L^{p-2} \underset{(t-\delta) \vee 0<s\le t+\delta}{\sup}\,
|\Delta J_s|^2+2^{p}\underset{0 \le s\le T}{\sup}\,|Y_s|^{p} \ind_{\{\underset{0 \le s\le T}{\sup}\,|Y_s|>L\}}.
\end{align*}

Therefore, in view of \eqref{J-2}, for any $L>0$ and $R>0$, 
\begin{align*}
\E\left[\underset{(t-\delta) \vee 0<t\le t+\delta}{\sup}\,|\Delta J_s|^p\right] & \le 2^{p-2} L^{p-2}\E\left[\int_{(t-\delta) \vee 0}^{t+\delta}|U_s|^2_{\nu}ds\right]+2^{p-2}\E\left[\underset{0 \le s\le T}{\sup}\,|Y_s|^{p} \ind_{\{\underset{0 \le s\le T}{\sup}\,|Y_s|>L\}}\right]\nonumber \\ & \le 2^{p-2} L^{p-2} \left(R^2 \delta+\E\left[\int_0^T|U_s|_{\nu}^2\ind_{\{|U_s|_{\nu}>R\}}ds\right]\right) \nonumber
\\ & +2^{p-2} \E\left[\underset{0 \le s\le T}{\sup}\,|Y_s|^{p} \ind_{\{\underset{0 \le s\le T}{\sup}\,|Y_s|>L\}}\right],
\end{align*}
which implies
$$
\begin{array}{lll}
\underset{0\le t\le T}{\sup}\,\E\left[\underset{(t-\delta) \vee 0<t\le t+\delta}{\sup}\,|\Delta J_s|^p\right] & \le 2^{p-2} L^{p-2} \left(R^2 \delta+\E\left[\int_0^T|U_s|_{\nu}^2\ind_{\{|U_s|_{\nu}>R\}}ds\right] \right) \\ & +2^{p-2} \E\left[\underset{0 \le s\le T}{\sup}\,|Y_s|^{p} \ind_{\{\underset{0 \le s\le T}{\sup}\,|Y_s|>L\}}\right].
\end{array}
$$
Thus,
$$\begin{array}{lll}
\underset{\delta\to 0}\lim\,\underset{0\le t\le T}{\sup}\E\left[\underset{(t-\delta) \vee 0<t\le t+\delta}{\sup}\,|\Delta J_s|^p\right] & \le
2^{p-2} L^{p-2} \E\left[\int_0^T|U_s|_{\nu}^2\ind_{\{|U_s|_{\nu}>R\}}ds\right]\\ & +2^{p-2}\E\left[\underset{0 \le s\le T}{\sup}\,|Y_s|^{p} \ind_{\{\underset{0 \le s\le T}{\sup}\,|Y_s|>L\}}\right].
\end{array}
$$
Since $Y \in \mathcal{S}^p$ and $U\in \mathcal{H}_{\nu}^p$, 
$$
\underset{R\to \infty}\lim\,\E\left[\int_0^T|U_s|_{\nu}^2\ind_{\{|U_s|_{\nu}>R\}} ds\right]=0,\quad \underset{L \to \infty}\lim\,\E\left[\underset{0 \le s\le T}{\sup}\,|Y_s|^{p} \ind_{\{\underset{0 \le s\le T}{\sup}\,|Y_s|>L\}}\right]=0.
$$
Hence, by taking first the limit in $R \to \infty$ and then $L \to \infty$, we obtain
$$
\underset{\delta\to 0}\lim\,\underset{0\le t\le T}{\sup}\,\E\left[\underset{(t-\delta) \vee 0<t\le t+\delta}{\sup}\,|\Delta J_s|^p\right]=0.
$$\qed
\end{proof}

We now provide the following convergence result  for the solution $Y^{i,n}$ of \eqref{BSDEParticle}.

\medskip
\begin{Proposition}[Convergence of the $Y^{i,n}$'s]\label{chaos-Y-1}
Assume that for some $p\ge 2$,  $\gamma_1$ and $\gamma_2$ satisfy 
\begin{align}\label{smallnessCond-chaos-1}
  2^{\frac{5p}{2}-2}(\gamma_1^p+\gamma_2^p)<1.
\end{align}
Then, under Assumptions \ref{generalAssump}, \ref{Assump-PS} and \ref{Assump:chaos}, we have
\begin{equation}\label{chaos-Y-1-1}
\underset{n\to\infty}\lim \,\underset{0\le t\le T}{\sup}\,E[|Y^{i,n}_t-Y^i_t|^p]=0.
\end{equation}
In particular,
\begin{equation}\label{chaos-Y-1-2}
\underset{n\to\infty}\lim \, \|Y^{i,n}-Y^i\|_{\mathcal{H}^{p,1}}=0.
\end{equation}
\end{Proposition}
\begin{proof}
For any $t\in[0,T]$, let $\vartheta \in \mathcal{T}^n_t$. By the estimates on BSDEs in Proposition \ref{estimates}, we have
\begin{equation}\label{Y-n-estimate}\begin{array} {lll}
|Y^{i,n}_\vartheta-Y^{i}_\vartheta|^{p} \\
\le \underset{\tau\in\mathcal{T}^n_\vartheta}{\esssup\,}\left|\mathcal{E}_{\vartheta,\tau}^{\textbf{f}^i \circ \textbf{Y}^{n}}[h(\tau,Y^{i,n}_\tau,L_{n}[\textbf{Y}^n_\tau])\ind_{\{\tau<T\}}+\xi^{i,n}\ind_{\{\tau=T\}}]- \mathcal{E}_{\vartheta,\tau}^{\textbf{f}^i \circ Y^i }[h(\tau,Y^i_\tau,\P_{Y_s|s=\tau})\ind_{\{\tau<T\}}+\xi^i \ind_{\{\tau=T\}}] \right|^{p} \\
\le \underset{\tau\in\mathcal{T}^n_\vartheta}{\esssup\,}2^{p/2-1}\E\left[\eta^p \left(\int_\vartheta^\tau e^{2 \beta (s-\vartheta)} |(\textbf{f}^i \circ \textbf{Y}^{n}) (s,\widehat{Y}^{i,\tau}_s,\widehat{Z}^{i,\tau}_s,\widehat{U}^{i,\tau}_s)-(\textbf{f}^i \circ Y^i) (s,\widehat{Y}^{i,\tau}_s, \widehat{Z}^{i,\tau}_s,\widehat{U}^{i,\tau}_s)|^2 ds\right)^{p/2} \right. \\ \left. \qquad\qquad \qquad +\left(e^{ \beta(\tau-\vartheta)}|h(\tau,Y^{i,n}_\tau,L_{n}[\textbf{Y}^n_\tau])-h(\tau,Y^i_\tau, \P_{Y_s|s=\tau})| + e^{\beta(T-\vartheta)}|\xi^{i,n}-\xi^i|\right)^p
|\mathcal{F}_\vartheta\right]\\
\le \underset{\tau\in\mathcal{T}^n_\vartheta}{\esssup\,}2^{p/2-1}\E\left[\int_\vartheta^{\tau}e^{p \beta (s-\vartheta)}T^{\frac{p-2}{2}}\eta^p C_f^{p}\mathcal{W}^p_p(L_{n}[\textbf{Y}_s^n],\P_{Y_s})ds +\left( \gamma_1 e^{ \beta(\tau-\vartheta)}|Y^{i,n}_{\tau}-Y^{i}_{\tau}| \right. \right. \\ \left.\left. \qquad\qquad\qquad+\gamma_2 e^{ \beta(\tau-\vartheta)} \mathcal{W}_p(L_{n}[\textbf{Y}_{\tau}^n],\P_{Y_s|s=\tau})  +e^{ \beta(T-\vartheta)}|\xi^{i,n}-\xi^i|\ind_{\{\tau=T\}}
\right)^p|\mathcal{F}_\vartheta\right],
 \end{array}
\end{equation}
where $\eta$, $\beta>0$ such that  $\eta \leq \frac{1}{C_f^2}$ and  $\beta \geq 2 C_f+\frac{3}{\eta}$, and $(\widehat{Y}^{i,\tau},\widehat{Z}^{i,\tau},\widehat{U}^{i,\tau})$ is the solution of the BSDE associated with driver $\textbf{f}^{i}\circ \Ybf^n$, terminal time $\tau$ and terminal condition $h(\tau,Y^{i,n}_\tau,L_{n}[\Ybf^n_s]_{s=\tau})\ind_{\{\tau<T\}}+\xi^{i,n}\ind_{\{\tau=T\}}$.

Therefore, we have
\begin{equation*}
e^{p\beta \vartheta}|Y^{i,n}_\vartheta-Y^{i}_\vartheta|^{p}\le \underset{\tau\in\mathcal{T}^n_\vartheta}{\esssup\,}\E[G^{i,n}_{\vartheta,\tau}|\mathcal{F}_\vartheta],
\end{equation*}
where
\begin{equation*}\begin{array} {lll}
G^{i,n}_{\vartheta,\tau}:=\int_\vartheta^{\tau}2^{3p/2-2}T^{\frac{p-2}{2}}\eta^p C_f^{p} \left(\frac{1}{n}\sum_{j=1}^ne^{p \beta s}|Y^{j,n}_s-Y^j_s|^p\right) ds
\\  \qquad\qquad\qquad +2^{\frac{5p}{2}-3}(\gamma_1^p+\gamma_2^p) \left(e^{p\beta\tau}|Y^{i,n}_{\tau}-Y^{i}_{\tau}|^p+
\frac{1}{n}\sum_{j=1}^ne^{p\beta\tau}|Y^{j,n}_{\tau}-Y^j_{\tau}|^p\right) 
\\  \qquad\qquad\qquad + 2^{p/2-1}\left(2^{2p-2}\gamma_2^p+2^{p-1}T^{p/2}\eta^p C_f^{p}\right)\underset{0\le t\le T}{\sup}\, e^{p \beta s}\mathcal{W}^p_p(L_{n}[\textbf{Y}_s],\P_{Y_s}) \\  \qquad\qquad\qquad + 2^{p/2-1}2^{2p-2}e^{\beta T}|\xi^{i,n}-\xi^i|^p\ind_{\{\tau=T\}}.
\end{array}
\end{equation*}
Setting
$V^{n,p}_t:=\frac{1}{n}\sum_{j=1}^ne^{p \beta t}|Y^{j,n}_t-Y^j_t|^p$
and 
$$
\Gamma_{n,p}:=2^{p/2-1}\left\{\left(2^{2p-2}\gamma_2^p+2^{p-1}T^{p/2}\eta^p C_f^{p}\right)\underset{0\le t\le T}{\sup}\, e^{p \beta s}\mathcal{W}^p_p(L_{n}[\textbf{Y}_s],\P_{Y_s})+ 2^{2p-2}e^{\beta T}|\xi^{i,n}-\xi^i|^p\ind_{\{\tau=T\}}\right\}.
$$
we obtain 
\begin{equation}\label{V-n-p}
V^{n,p}_\vartheta\le \underset{\tau\in\mathcal{T}^n_\vartheta}{\esssup\,}\E[\int_\vartheta^{\tau} 2^{3p/2-2}\eta^p T^{\frac{p-2}{2}} C_f^{p} V^{n,p}_sds+2^{5p/2-2}(\gamma_1^p+\gamma_2^p)V^{n,p}_{\tau}+\Gamma_{n,p}|\mathcal{F}_\vartheta].
\end{equation}
In view of Lemma D.1 in \cite{KS98}, we have 
$$
\E[V^{n,p}_\vartheta]\le 2^{5p/2-2}(\gamma_1^p+\gamma_2^p)\underset{\tau\in\mathcal{T}^n_\vartheta}{\sup}\,\E[V^{n,p}_{\tau}]+\E[\int_\vartheta^T 2^{3p/2-2}\eta^p C_f^{p} V^{n,p}_sds+\Gamma_{n,p}].
$$
Since for any $\vartheta \in \mathcal{T}^n_t$, $\mathcal{T}^n_\vartheta \subset \mathcal{T}^n_t$, we have
$$
\lambda \underset{t\le s\le T}{\sup}\,\E[V^{n,p}_{s}] \leq \lambda \underset{ \vartheta \in \mathcal{T}_t^n}{\sup}\,\E[V^{n,p}_{\vartheta}]\le \E[\int_t^T 2^{3p/2-2}\eta^p C_f^{p} V^{n,p}_sds+\Gamma_{n,p}].
$$
In particular,
$$
\lambda\E[V^{n,p}_{t}]\le \E[\int_t^T 2^{3p/2-2}\eta^p C_f^{p} V^{n,p}_sds+\Gamma_{n,p}],
$$
where $\lambda:=1-2^{5p/2-2}(\gamma_1^p+\gamma_2^p) >0$ by the assumption \eqref{smallnessCond-chaos-1}. Now, by the backward Gronwall inequality (see Proposition \ref{gronwall} below), applied to $g(t):=\E[V^{n,p}_{t}]$ which is bounded since $V^{n,p}\in\mathcal{S}_1$,  $c:=\E[\Gamma_{n,p}]/\lambda$ and $\alpha(t):=\frac{1}{\lambda}2^{3p/2-2}\eta^p C_f^{p}$, we obtain
$$
\E[V^{n,p}_{t}]\le  \frac{e^{K_p}}{\lambda}\E[\Gamma_{n,p}]
$$
where $K_p:=\frac{1}{\lambda}2^{3p/2-2}\eta^p C_f^{p}T$. Hence, since the choice of $t\in[0,T]$ is arbitrary, we must have
$$
\underset{0\le t\le T}{\sup}\,\E[V^{n,p}_{t}]\le  \frac{e^{K_p}}{\lambda}\E[\Gamma_{n,p}].
$$

But, in view of the exchangeability of the processes $(Y^{i,n},Y^i),i=1,\ldots,n$ (see Proposition \ref{exchange}), we have, $\E[V^{n,p}_{t}]=\E[e^{p\beta t}|Y^{i,n}_t-Y^i_t|^p]$. Thus,
$$
\underset{0\le t\le T}{\sup}\,\E[e^{p\beta t}|Y^{i,n}_t-Y^i_t|^p]\le \frac{e^{K_p}}{\lambda}\E[\Gamma_{n,p}] \to 0
$$
as $n\to\infty$,
in view of Theorem \ref{LLN-2} and Assumption \ref{Assump:chaos}, as required.  \qed
\end{proof}
\begin{Theorem}[Propagation of chaos of the $Y^{i,n}$'s]\label{prop-y} 
Under the assumptions of Proposition \ref{chaos-Y-1}, the solution $Y^{i,n}$ of the particle system \eqref{BSDEParticle} satisfies the propagation of chaos property, i.e. for any fixed positive integer $k$, 
$$
\underset{n\to\infty}{\lim} \text{Law}\,(Y^{1,n},Y^{2,n},\ldots,Y^{k,n})=\text{Law}\,(Y^1,Y^2,\ldots,Y^k).
$$
\end{Theorem}

\begin{proof}
 Set $P^{k,n}:=\text{law}\,(Y^{1,n},Y^{2,n},\ldots,Y^{k,n})$ and $P^{\otimes k}:=\text{Law}\,(Y^1,Y^2,\ldots,Y^k)$.
Consider the Wasserstein metric on $\mathcal{P}_2(\mathbb{H}^2)$ defined by 
\begin{equation}\label{W-Y} 
D_{\mathbb{H}^2}(P,Q)=\inf\left\{\left(\int_{\mathbb{H}^2\times \mathbb{H}^2}  \|y-y^{\prime}\|_{\mathbb{H}^2}^2 R(dy,dy^{\prime})\right)^{1/2}\right\},
\end{equation}
over $R\in\mathcal{P}(\mathbb{H}^2\times \mathbb{H}^2)$ with marginals $P$ and $Q$. Note that, since $p \geq 2$, it is enough to show for $D_{\mathbb{H}^2}(P,Q)$. Since $\mathbb{H}^2$ is a Polish space, $(\mathcal{P}_2(\mathbb{H}^2), D_{\mathbb{H}^2})$ is a Polish space and induces the topology of weak convergence (see Theorem 6.9 in \cite{Villani}). Thus, we obtain the propagation of chaos property for the $Y^{i,n}$'s if we can show that 
$\underset{n\to\infty}{\lim}D_{\mathbb{H}^2}(P^{k,n},P^{\otimes k})=0$. But, this follows from the fact that 
$$
D^2_{\mathbb{H}^2}(P^{k,n},P^{\otimes k})\le k\sup_{i\le k}\|Y^{i,n}-Y^i\|_{\mathcal{H}^{2,1}}^2
$$
and \eqref{chaos-Y-1-2} for $p=2$. \qed

\end{proof}

\medskip
In the next proposition we provide a convergence result for the whole solution $(Y^{i,n},Z^{i,n},U^{i,n},K^{i,n})$ of \eqref{BSDEParticle}.

\begin{Proposition}\label{chaos-1} Assume that, for some $p>2$, $\gamma_1$ and $\gamma_2$ satisfy 
\begin{align}\label{smallnessCond-chaos}
  2^{5p/2-2}(\gamma_1^p+\gamma_2^p)<\left(\frac{p-\kappa}{2p}\right)^{p/\kappa}
\end{align}
for some $\kappa\in [2,p)$. 

Then, under Assumptions \ref{generalAssump}, \ref{Assump-PS} and \ref{Assump:chaos}, we have
\begin{equation}\label{chaos-1-1}\begin{array}{ll}
\underset{n\to\infty}\lim \,\left( \|Y^{i,n}-Y^i\|_{\sp}+ \|Z^{i,n}-Z^i{\bf e}_i\|_{\mathcal{H}^{p,n}}+ \|U^{i,n}-U^i{\bf e}_i\|_{\mathcal{H}^{p,n}_\nu}+ \|K^{i,n}-K^i\|_{\sp}\right)=0,
\end{array}
\end{equation}
where ${\bf e}_1,\ldots,{\bf e}_n$ are unit vectors in $\R^n$.
\end{Proposition}

\begin{Remark}  $ $
\begin{itemize}
    \item If the obstacle $h$ is independent of the solution $Y$, the limit result \eqref{chaos-1-1} remains valid for $p\ge 2$ without the condition \eqref{smallnessCond-chaos}.
\end{itemize}
\end{Remark}

\begin{proof} 
\uline{Step 1.}   Let us first show that $\underset{n\to\infty}{\lim}\|Y^{i,n}-Y^i\|_{\mathcal{S}^p_{\beta}}=0$.

\medskip
Let $u \in [0,T]$. In view of \eqref{Y-n-estimate}, for any $\kappa \ge 2$ and any $u \leq t\leq T$, we have
\begin{equation*}\begin{array} {lll}
|Y^{i,n}_t-Y^{i}_t|^{\kappa} 
\le \underset{\tau\in\mathcal{T}^n_t}{\esssup\,}\,\, 2^{\kappa/2-1}\E\left[\int_t^{\tau}e^{\kappa \beta (s-t)}T^{\frac{\kappa-2}{2}}\eta^{\kappa} C_f^{\kappa}\mathcal{W}^{\kappa}_{p}(L_{n}[\textbf{Y}_s^n],\P_{Y_s})ds +\left( 
\gamma_1 e^{ \beta(\tau-t)}|Y^{i,n}_{\tau}-Y^{i}_{\tau}| \right. \right. \\ \left.\left. \qquad\qquad\qquad\qquad\qquad+\gamma_2 e^{ \beta(\tau-t)} \mathcal{W}_{p}(L_{n}[\textbf{Y}_{\tau}^n],\P_{Y_s|s=\tau})  +e^{ \beta(T-t)}|\xi^{i,n}-\xi^i|\ind_{\{\tau=T\}}
\right)^{\kappa}|\mathcal{F}_t\right],
 \end{array}
\end{equation*}
where $\eta$, $\beta>0$ such that  $\eta \leq \frac{1}{C_f^2}$ and  $\beta \geq 2 C_f+\frac{3}{\eta}$.

Therefore, for any $p>\kappa \ge 2$, we have
\begin{equation*}
e^{p\beta t}|Y^{i,n}_t-Y^{i}_t|^{p}\le \E[\mathcal{G}^{i,n}_{u,T}|\mathcal{F}_t]^{p/\kappa},
\end{equation*}
where
$$
\begin{array}{lll}
\mathcal{G}_{u,T}^{i,n}:=2^{\kappa/2-1}\left(\int_u^T T^{\frac{\kappa-2}{2}}e^{\kappa \beta s}\eta^{\kappa} C_f^{\kappa}\mathcal{W}^{\kappa}_{p}(L_{n}[\textbf{Y}_s^n],\P_{Y_s})ds \right. \\ \left. \qquad\qquad\qquad +\left( 
\gamma_1 \underset{u\le s\le T}{\sup}\,e^{\beta s}|Y^{i,n}_{s}-Y^{i}_{s}| +\gamma_2 \underset{u\le s\le T}{\sup}\,e^{ \beta s} \mathcal{W}_{p}(L_{n}[\textbf{Y}_{s}^n],\P_{Y_s}) +e^{ \beta T}|\xi^{i,n}-\xi^i|\right)^{\kappa}\right).
\end{array}
$$
Thus, by Doob's inequality, we get
\begin{equation}\label{Y-n-Doob}
\E[\underset{u\le t\le T}{\sup}\,e^{p\beta t}|Y^{i,n}_t-Y^{i}_t|^{p}]\le \left(\frac{p}{p-\kappa}\right)^{p/\kappa}\E\left[\left(\mathcal{G}_{u,T}^{i,n}\right)^{p/\kappa}\right].
\end{equation}

Therefore, we have
$$\begin{array}{lll}
2^{1-\frac{p}{2}}\left(\mathcal{G}_{u,T}^{i,n}\right)^{p/\kappa}\le C_1\int_u^T \underset{s\le t\le T}{\sup}\,e^{p\beta t}\mathcal{W}_p^p(L_{n}[\textbf{Y}_{t}^n],L_{n}[\textbf{Y}_{t}])ds \\ \qquad\qquad\qquad\qquad +2^{p-1}\left(\gamma_1\underset{u\le t\le T}{\sup}\,e^{\beta t}|Y^{i,n}_t-Y^{i}_t|+\gamma_2\underset{u\le t\le T}{\sup}\,e^{\beta t}\mathcal{W}_{p}(L_{n}[\textbf{Y}_{s}^n],L_{n}[\textbf{Y}_{s}])\right)^p +\Lambda_n(u)
\end{array}
$$
where $C_1:=2^{p-1}T^{\frac{p}{2}}\eta^pC_f^p$ and
$$
\Lambda_n:=C_1 \underset{0\le s\le T}{\sup}\,e^{ p\beta s} \mathcal{W}_{p}(L_{n}[\textbf{Y}_{s}],\P_{Y_s})+2^{p-1}\left(
\gamma_2 \underset{0\le s\le T}{\sup}\,e^{ \beta s} \mathcal{W}_{p}(L_{n}[\textbf{Y}_{s}],\P_{Y_s}) +e^{ \beta T}|\xi^{i,n}-\xi^i|\right)^{p}.
$$
\noindent But, in view of the exchangeability of the processes $(Y^{i,n},Y^i),i=1,\ldots,n$ (see Proposition \ref{exchange}), for each $s \in [0,T]$ and $u \in [0,T]$ we have, 
$$
\E\left[\underset{u\le t\le s}{\sup}e^{\beta p t}\mathcal{W}_p^p(L_n[\textbf{Y}_t^n],L_n[\textbf{Y}_t])\right]\leq \E\left[\underset{u\le t\le s}{\sup}e^{\beta p t}|Y^{i,n}_t-Y^{i}_t|^p\right],\quad i=1,\ldots,n,
$$
and so
$$\begin{array}{lll}
2^{1-\frac{p}{2}}\E[\left(\mathcal{G}_{u,T}^{i,n}\right)^{p/\kappa}]\le C_1\E\left[\int_u^T \underset{s\le t\le T}{\sup}\,e^{p\beta t}\mathcal{W}_p^p(L_{n}[\textbf{Y}_{t}^n],L_{n}[\textbf{Y}_{t}])ds\right]+\E\left[\Lambda_n\right]\\ \qquad\qquad\qquad\qquad\qquad\qquad
+2^{2p-1}(\gamma_1^p+\gamma_2^p)\E[\underset{u\le t\le T}{\sup}\,e^{p\beta t}|Y^{i,n}_t-Y^{i}_t|^{p}].
\end{array}
$$
Therefore, \eqref{Y-n-Doob} becomes
\begin{equation*}
\begin{array} {lll}
\mu\E\left[\underset{u\le t\le
T}{\sup}e^{p\beta t}|Y^{i,n}_t-Y^{i}_t|^p\right] \le C_1\E\left[\int_u^T \underset{s\le t\le T}{\sup}e^{\beta p t}|Y^{i,n}_t-Y^{i}_t|^pds\right]+\E\left[\Lambda_n\right].
\end{array}
\end{equation*}
where $\mu:=2^{1-\frac{p}{2}}\left(\frac{p}{p-\kappa}\right)^{-p/\kappa}-2^{2p-1}(\gamma_1^p+\gamma_2^p)$.
Using the condition \eqref{smallnessCond-chaos}, to see that
$\mu>0$, and the backward Gronwall inequality (see Proposition \ref{gronwall} below), we finally  obtain
\begin{equation*}
\E\left[\underset{0\le t\le
T}{\sup}e^{p\beta t}|Y^{i,n}_t-Y^{i}_t|^p\right] \le \frac{e^{\frac{C_1}{\mu}T}}{\mu}\E\left[\Lambda_n\right].
\end{equation*}
Next, by \eqref{GC-2} together with Assumption \eqref{Assump:chaos} (iii) we have
$$
\underset{n\to\infty}\lim\, \E\left[\Lambda_n\right]= 0,
$$
which yields the desired result.

\medskip
\noindent \uline{Step 2.} We now show that $\underset{n\to\infty}{\lim}\|Z^{i,n}-Z^i{\bf e}_i\|_{\mathcal{H}^{p,n}}=0,    \,\underset{n\to\infty}{\lim}\|U^{i,n}-U^i{\bf e}_i\|_{\mathcal{H}^{p,n}_\nu}=0$ and $\underset{n\to\infty}{\lim}\|K^{i,n}-K^i\|_{\mathcal{S}^{p}}=0$. We first prove that $\underset{n\to\infty}{\lim}\|Z^{i,n}-Z^i{\bf e}_i\|_{\mathcal{H}^{p,n}}=0$ and    $\underset{n\to\infty}{\lim}\|U^{i,n}-U^i{\bf e}_i\|_{\mathcal{H}^{p,n}_\nu}=0$. For $s \in [0,T]$, denote $\delta Y^{i,n}_s:= Y_s^{i,n}-Y_s^{i}$, $\delta Z_s^{i,n}:= Z_s^{i,n}-Z_s^{i}{\bf e}_i$, $\delta U^{i,n}_s:= U_s^{i,n}-U_s^{i}{\bf e}_i$, $\delta K_s^{i,n}:= K_s^{i,n}-K_s^{i}$, $\delta f^{i,n}_s:=f(s,Y_s^{i,n}, Z_s^{i,i,n}, U_s^{i,i,n}, L_{n}[\textbf{Y}_s^n])-f(s,Y_s^{i}, Z_s^{i}, U_s^{i}, \mathbb{P}_{Y^{i}_s})$, $\delta \xi^{i,n}:=\xi^{i,n}-\xi^{i}$ and $\delta h^{i,n}_s:=h(s,Y_s^{i,n},L_{n}[\textbf{Y}^n_s])-h(s,Y_s^{i}, \mathbb{P}_{Y^{i}_s})$. By applying It\^{o}'s formula to $|\delta Y^{i,n}_t|^2$, we obtain
\begin{align*}
    &|\delta Y^{i,n}_t|^2+\int_t^T |\delta Z^{i,n}_s|^2ds+\int_t^T\int_{\R^*}\sum_{j=1}^n|\delta U^{i,j,n}_s(e)|^2 N^j(ds,de)+\sum_{t<s\leq T}|\Delta K_s^{i,n}-\Delta K_s^{i}|^2=  |\delta \xi^{i,n}|^2 \nonumber \\
    &  + 2\int_t^T \delta Y_s^{i,n}\delta f_s^{i,n}ds - 2 \int_t^T \delta Y_s^{i,n} \sum_{j=1}^n\delta Z_s^{i,j,n} dB^j_s - 2 \int_t^T \int_{\R^*} \delta Y_s^{i,n} \sum_{j=1}^n\delta U_s^{i,j,n}(e) \Tilde{N}^j(ds,de) \nonumber \\ & +2\int_t^T \delta Y_s^{i,n} d(\delta K_s^{i,n}).
\end{align*}
By standard estimates, from the assumptions on the driver $f$, we get, for all $\varepsilon>0$,
\begin{equation*}\begin{array}{lll}
    \int_t^T\delta Y_s^{i,n}\delta f_s^{i,n}ds \leq \int_t^T C_f|\delta Y_s^{i,n}|^2ds+\int_t^T\frac{3}{\varepsilon}C^2_f|\delta Y_s^{i,n}|^2ds+\int_t^T\frac{\varepsilon}{4} \left\{|\delta Z^{i,n}_s|^2+|\delta U^{i,n}_s|^2_\nu\right\}ds \\ \qquad\qquad\qquad\qquad\qquad
    +\int_t^T\frac{\varepsilon}{4}\mathcal{W}^2_p(L_{n}[\textbf{Y}^n_s],\mathbb{P}_{Y^{i}_s}) ds.
    \end{array}
\end{equation*}
It follows that, for a given constant $C_p>0$, independent of $n$,  we have
\begin{equation}\begin{array}{lll}\label{eq1}
    |\delta Y^{i,n}_0|^p+\left(\int_0^T |\delta Z^{i,n}_s|^2ds\right)^{\frac{p}{2}}+\left(\int_0^T\int_{\R^*}\sum_{j=1}^n|\delta U^{i,j,n}_s(e)|^2 N^j(ds,de)\right)^{\frac{p}{2}} \leq C_p |\delta \xi^{i,n}|^p  \\ \qquad +C_p \left\{\left(2C_f+\frac{6}{\varepsilon}C_f^2\right)T\right\}^{\frac{p}{2}} \underset{0\le s\le T}{\sup} |\delta Y_s^{i,n}|^p  
    +C_p \varepsilon^{\frac{p}{2}} \left(\int_0^T |\delta Z^{i,n}_s|^2 ds\right)^{\frac{p}{2}}+C_p \varepsilon^{\frac{p}{2}} \left(\int_0^T |\delta U^{i,n}_s|^2_\nu ds\right)^{\frac{p}{2}} \\ \qquad +C_p \left(\int_0^T \mathcal{W}^2_p(L_{n}[\textbf{Y}^n_s],\mathbb{P}_{Y^{i}_s})ds\right)^{\frac{p}{2}} +C_p \left\{ \left|\int_0^T \delta Y_s^{i,n} \sum_{j=1}^n\delta Z_s^{i,j,n} dB^j_s \right|^{\frac{p}{2}} 
    \right. \\  \left. \qquad + \left|\int_0^T \int_{\R^*} \delta Y_s^{i,n} \sum_{j=1}^n\delta U_s^{i,j,n}(e) \Tilde{N}^j(ds,de) \right|^{\frac{p}{2}}  +\left|\int_0^T \delta Y_s^{i,n} d(\delta K_s^{i,n})\right|^{\frac{p}{2}}\right \}.
    \end{array}
\end{equation}
By applying the Burkholder-Davis-Gundy inequality, we derive that there exist some constants $m_p>0$ and $l_p>0$ such that
\begin{align*}
C_p \mathbb{E}\left[\left|\int_0^T \delta Y_s^{i,n} \sum_{j=1}^n\delta Z_s^{i,j,n} dB^j_s \right|^{\frac{p}{2}}\right]  & \leq m_p \mathbb{E} \left[\left(\int_0^T (\delta Y_s^{i,n})^2 |\delta Z_s^{i,n}|^2 ds\right)^{\frac{p}{4}}\right] \\ & \leq \frac{m^2_p}{4} \|\delta Y_s^{i,n}\|^p_{\mathcal{S}^p}+\frac{1}{2}\mathbb{E}\left[\left(\int_0^T |\delta Z^{i,n}_s|^2ds\right)^{\frac{p}{2}}\right]
\end{align*}
and
\begin{align*}
C_p \mathbb{E}\left[\left|\int_0^T \int_{\R^*} \delta Y_s^{i,n} \sum_{j=1}^n \delta U_s^{i,j,n}(e) \Tilde{N}^j(ds,de) \right|^{\frac{p}{2}}\right] \leq l_p \mathbb{E} \left[\left(\int_0^T (\delta Y_s^{i,n})^2 \int_{\R^*}\sum_{j=1}^n(\delta U_s^{i,j,n}(e))^2 N^j(ds,de)\right)^{\frac{p}{4}}\right] \nonumber \\ \qquad\qquad \leq \frac{l^2_p}{4} \|\delta Y_s^{i,n}\|^p_{\mathcal{S}^p}+\frac{1}{2}\mathbb{E}\left[\left(\int_0^T \int_{\R^*}\sum_{j=1}^n(\delta U^{i,j,n}_s(e))^2 N^j(ds,de)\right)^{\frac{p}{2}}\right].
\end{align*}
Also recall that, for some constant $e_p>0$ we have (cf. Eq. (1.3) in \cite{MM14})
\begin{align*}
    \mathbb{E}\left[\left(\int_0^T \int_{\R^*}|\delta U^{i,n}_s(e)|^2\nu(de)ds\right)^{\frac{p}{2}}\right] \leq e_p \mathbb{E}\left[\left(\int_0^T \int_{\R^*} \sum_{j=1}^n |\delta U^{i,j,n}_s(e)|^2N^j(de,ds)\right)^{\frac{p}{2}}\right].
\end{align*}
Now, we take the expectation in $\eqref{eq1}$, by using the above inequalities and taking $\varepsilon>0$ small enough, we obtain
\begin{align*}
    \mathbb{E}\left[\left(\int_0^T |\delta Z^{i,n}_s|^2ds\right)^{\frac{p}{2}}+\left(\int_0^T \|\delta U^{i,n}_s\|_\nu^2ds\right)^{\frac{p}{2}}\right]  \leq  C_p |\delta \xi^{i,n}|^p+K_{C_f,\varepsilon,T,p} \|\delta Y_s^{i,n}\|^p_{\mathcal{S}^p} \nonumber \\ \qquad\qquad\qquad\qquad +C_p \mathbb{E}\left[\sup_{0 \leq s \leq T} \mathcal{W}^p_p(L_{n}[\textbf{Y}^n_s],\mathbb{P}_{Y^{i}_s})\right]+\mathbb{E}\left[\left(\sup_{0 \leq s \leq T}|\delta Y_s^{i,n}|(K_T^{i,n}+K_T^{i})\right)^{\frac{p}{2}}\right].
\end{align*}
The above inequality, together with the equations
\begin{align}\label{eq2}
    K_T^{i,n}=Y_0^{i,n}-\xi^{i,n}& -\int_0^Tf(s,Y_s^{i,n}, Z_s^{i,i,n}, U_s^{i,i,n}, L_{n}[\textbf{Y}_s^n])ds \nonumber \\ & +\sum_{j=1}^n\int_0^TZ_s^{i,j,n}dB^j_s+\int_0^T\int_{\R^*} \sum_{j=1}^nU_s^{i,j,n}(e)\Tilde{N}^j(ds,de)
\end{align}
and
\begin{align}\label{eq3}
K_T^{i}=Y_0^{i}-\xi^{i}-\int_0^Tf(s,Y_s^{i}, Z_s^{i}, U_s^{i}, \mathbb{P}_{Y_s^{i}})ds+\int_0^TZ_s^{i}dB^i_s+\int_0^T\int_{\R^*} U_s^{i}(e)\Tilde{N}^i(ds,de),
\end{align}
lead to
\begin{align}\label{ineq}
    & \mathbb{E}\left[\left(\int_0^T |\delta Z^{i,n}_s|^2ds\right)^{\frac{p}{2}}+\left(\int_0^T \int_{\R^*}|\delta U^{i,n}_s(e)|^2 \nu(de)ds\right)^{\frac{p}{2}}\right]  \nonumber \\ & \leq C_p \left( \|\delta Y^{i,n}\|^p_{\mathcal{S}^p}+\mathbb{E}[\sup_{0 \leq s \leq T} \mathcal{W}^p_p(L_{n}[\textbf{Y}^n_s],\mathbb{P}_{Y^{i}_s})]+|\delta \xi^{i,n}|^p+ \|\delta Y^{i,n}\|_{\mathcal{S}^p}^p \Psi^{i,n}_p\right),
\end{align}
with 
$$
\Psi^{i,n}_p:=\mathbb{E}\left[\|Y^{i,n}\|^p_{\mathcal{S}_p}+\|Y^{i}\|^p_{\mathcal{S}_p}+|\xi^{i,n}|^p+|\xi^{i}|^p+\left(\int_0^T|f(s,0,0,0,\delta_0)|^2ds\right)^{p/2}+\sup_{0 \leq s \leq T}|h(s,0,\delta_0)|^p\right],
$$
for a different constant $C_p$.

From Step 1, we have $\|\delta Y^{i,n}\|^p_{\mathcal{S}^p} \to 0$, which also implies the uniform boundedness of the sequence $\left(\|Y^{i,n}\|^p_{\mathcal{S}^p}\right)_{n \geq 0}$, and in particular the one of  $\Psi^{i,n}_p$. This result together with Assumption \ref{Assump:chaos} and inequality \eqref{ineq} gives the desired convergence.
In order to show that $\underset{n\to\infty}{\lim}\|K^{i,n}-K^i\|_{\mathcal{S}^{p}}=0$, we use the equalities  $\eqref{eq2}$ and $\eqref{eq3}$ and the convergence of $(Y^{i,n},Z^{i,n}, U^{i,n})$ shown above.\qed
\end{proof}

\begin{Theorem}[Propagation of chaos of the whole solution]\label{prop-y-u-z} 
Under the assumptions of Proposition \ref{chaos-1}, the particle system \eqref{BSDEParticle} satisfies the propagation of chaos property, i.e. for any fixed positive integer $k$, 
$$
\underset{n\to\infty}{\lim} \text{Law}\,(\Theta^{1,n},\Theta^{2,n},\ldots,\Theta^{k,n})=\text{Law}\,(\Theta^1,\Theta^2,\ldots,\Theta^k).
$$
\end{Theorem}
\begin{proof}
We obtain the propagation of chaos if we can show that 
$\underset{n\to\infty}{\lim}D_G^p(\P^{k,n},\P_{\Theta}^{\otimes k})=0$. But, in view of the inequality \eqref{chaos-0}, it suffices to show that for each $i=1,\ldots,k$,
\begin{align*}
    \underset{n\to\infty}\lim \,\left( \|Y^{i,n}-Y^i\|_{\sp}+ \|Z^{i,n}-Z^i{\bf e}_i\|_{\mathcal{H}^{p,n}}+ \|U^{i,n}-U^i{\bf e}_i\|_{\mathcal{H}^{p,n}_{\nu}}+ \|K^{i,n}-K^i\|_{\sp}\right)=0,
\end{align*}
which is derived in Proposition \ref{chaos-1}. \qed
\end{proof}

\begin{Remark}\label{omega-f-h} 
All the results of Sections 3 and 4 extend to the case where 
\begin{itemize} 
\item $f(t,\omega,y,z,\mu):=\hat{f}(t,(B_s(\omega),\tilde{N}_s(\omega))_{0\leq s \leq t},y,z,\mu)$, with $\hat{f}$ measurable and Lipschitz w.r.t. $y,z,\mu$.
\item $h(t,\omega,y,\mu):=\hat{h}(t,(B_s(\omega),\tilde{N}_s(\omega))_{0\leq s \leq t},y,\mu)$, with $\hat{h}$ measurable and Lipschitz w.r.t. $y,\mu$.
\item $(f^{i} \circ \textbf{Y}^n)(t,\omega,y,z,u)=\hat{f}(t,(B^i_s(\omega),\tilde{N}^i_s(\omega))_{0\leq s \leq t},y,z^{i},u^{i},\textbf{L}_n[\textbf{Y}_t^n](\omega))$.
\item $(f^{i} \circ {Y})(t,\omega,y,z,u):=\hat{f}(t,(B^i_s(\omega),\tilde{N}^i_s(\omega))_{0\leq s \leq t},y,z^{i},u^{i},\mathbb{P}_{Y_t})$.
\item $h(t,\omega,Y_t^{i,n}(\omega),\textbf{L}_n[\textbf{Y}_t^n](\omega))):=\hat{h}(t,(B^i_s(\omega),\tilde{N}^i_s(\omega))_{0\leq s \leq t},Y_t^{i,n}(\omega),\textbf{L}_n[\textbf{Y}_t^n](\omega))$.
\item $h(t,\omega,Y_t^{i}(\omega),\mathbb{P}_{Y^i_t})):=\hat{h}(t,(B^i_s(\omega),\tilde{N}^i_s(\omega))_{0\leq s \leq t},Y_t^{i}(\omega),\mathbb{P}_{Y^i_t}).$
\end{itemize}
\end{Remark}

\medskip


\appendix 
\section{Technical results}
We establish here the following $L^p$ estimates with \textit{universal constants} for the difference of the solutions $(Y^{i}, Z^{i}, U^{i})$, $i=1,2$ of BSDEs.

\begin{Proposition}[$L^p$ a priori estimates with \textit{universal constants}]\label{estimates} Let $T>0$. Let $\tau$ be a $\mathbb{F}$-stopping time with values in $[0,T]$. Let $p\ge 2$ and let $\xi_1$ and $\xi_2$ $\in L^p(\mathcal{F}_\tau)$. Let $f_1$ be a Lipschitz driver with constant $C$ and let $f_2$ be a driver. For $i=1,2$, let $(Y^{i},Z^{i}, U^{i})$ be a solution of the BSDE associated to terminal time $\tau$, driver $f^{i}$, and terminal condition $\xi^{i}$. For $s \in [0,\tau]$ denote $\delta Y_s:= Y_s^{1}-Y_s^{2}$, $\delta Z_s:= Z_s^{1}-Z_s^{2}$, $\delta U_s:= U_s^{1}-U_s^{2}$, $\delta f_s:=f^{1}(s,Y_s^{2}, Z_s^{2}, U_s^{2})-f^{2}(s,Y_s^{2}, Z_s^{2}, U_s^{2})$ and $\delta \xi:=\xi^1-\xi^{2}$. Let $\eta, \beta>0$ be such that $\beta \geq 2C+\frac{3}{\eta}$ and $\eta \leq \frac{1}{ C^2}$, then for each $t \in [0,\tau]$ we have
\begin{align}
    |e^{\beta t} \delta Y_t|^{p} \leq  2^{p/2-1}\left(\mathbb{E}\left[|e^{\beta \tau} \delta \xi|^p|\mathcal{F}_t\right]+\eta^p \mathbb{E}\left[\left(\int_t^{\tau} |e^{\beta s} \delta f_s|^2 ds \right)^{p/2}|\mathcal{F}_t\right]\right) \,\,\, \text{a.s.}
\end{align}
\end{Proposition}
\begin{proof}
Recall that, by standard estimates obtained by applying It\^o's formula to $e^{\beta t}|\delta Y_t|^2$, we derive that for $\beta \geq \frac{3}{\eta}+2C$ and $\eta \leq \frac{1}{C^2}$ (see \cite{dqs15}), we have
\begin{align}
    |e^{\beta t} \delta Y_t|^{2} \leq \mathbb{E}\left[|e^{\beta \tau} \delta \xi|^2|\mathcal{F}_t\right]+\eta^2 \mathbb{E}\left[\left(\int_t^{\tau} |e^{\beta s} \delta f_s|^2 ds \right)|\mathcal{F}_t\right] \,\,\, \text{a.s.},
\end{align}
from which it follows that 
\begin{align}
    |e^{\beta t} \delta Y_t|^{p} \leq \left(\mathbb{E}\left[|e^{\beta \tau} \delta \xi|^2|\mathcal{F}_t\right]+\eta^2 \mathbb{E}\left[\left(\int_t^{\tau} |e^{\beta s} \delta f_s|^2 ds \right)|\mathcal{F}_t\right]\right)^{p/2} \,\,\, \text{a.s.},
\end{align}
which leads to, by convexity relation and H{\"o}lder inequality,
\begin{align}
    |e^{\beta t} \delta Y_t|^{p} \leq 2^{p/2-1}\left(\mathbb{E}\left[|e^{\beta \tau} \delta \xi|^p|\mathcal{F}_t\right]+\eta^p\mathbb{E}\left[\left(\int_t^{\tau} |e^{\beta s} \delta f_s|^2 ds \right)^{p/2}|\mathcal{F}_t\right]\right) \,\,\, \text{a.s.}.
\end{align}
\end{proof}

In the next proposition we recall a backward version of the celebrated Gronwall's inequality, we used in Section \ref{convergence}. We include a proof for the sake of completeness. 
\begin{Proposition}[Backward Gronwall inequality]\label{gronwall}
Let $\alpha$ be a non-negative Borel function such that $\int_0^T \alpha(t) dt <+\infty$. Let $c\ge 0$ and let $g$ be a Borel function that is bounded on $[0,T]$ and satisfies
\begin{equation}\label{gron-1}
0\le g(t)\le c+\int_t^T \alpha(s)g(s)\,ds,\quad t\in[0,T].
\end{equation}
Then
\begin{equation}\label{gron-2}
g(t)\le c \exp{\left(\int_t^T \alpha(s)\,ds\right)}.
\end{equation}
\end{Proposition}
\begin{proof}

Consider the Borel measure $\nu(dt):=\alpha(t)dt$ on $[0,T]$. Iterating \eqref{gron-1} yields
$$\begin{array}{lll}
g(t)\le c+c \sum_{k=1}^{\infty} \int_{[t,T)} \int_{[t_1,T)} \cdots \int_{[t_{k-1},T)}\nu(dt_k)\,\nu(dt_{k-1})\cdots \nu(dt_1) \\ \qquad
\le c+c \sum_{k=1}^{\infty} \frac{1}{k!} (\nu([t,T)))^k \\
\qquad = ce^{\nu([t,T))}=c\exp{\left(\int_t^T \alpha(s)\,ds\right)}.
\end{array}
$$

\end{proof}

{}


\begin{thebibliography}{}

\bibitem{AD11} Andersson, D. and Djehiche, B. (2011). A maximum principle for SDEs of mean-field type. Applied Mathematics \& Optimization, 63(3), 341--356.
  
\bibitem{abc19} Acciaio, B., Backhoff-Veraguas, J., and Carmona, R. (2019). Extended mean field control problems: stochastic maximum principle and transport perspective. SIAM journal on Control and Optimization, 57(6), 3666-3693.

\bibitem{bbp97} Barles, G., Buckdahn, R., and Pardoux, E. (1997). Backward stochastic differential equations and integral-partial differential equations. Stochastics: An International Journal of Probability and Stochastic Processes, 60(1-2), 57-83.

\bibitem{b68} Billingsley, P. (1968). Convergence of probability measures. John Wiley \& Sons.

\bibitem{bcdph20} Briand, P., Cardaliaguet, P., de Raynal, P-\'{E} C and Hu, Y. (2020). Forward and backward stochastic differential equations with normal constraints in law. Stochastic Processes and their Applications, 130, 7021-7097.

\bibitem{bh21} Briand, P., and Hibon, H. (2021). Particles Systems for mean reflected BSDEs. Stochastic Processes and their Applications, 131, 253-275.

\bibitem{blp09} Buckdahn, R., Li, J., and Peng, S. (2009). Mean-field backward stochastic differential equations and related partial differential equations. Stochastic processes and their Applications, 119(10), 3133-3154.

\bibitem{bdlp09} Buckdahn, R., Djehiche, B., Li, J., and Peng, S. (2009). Mean-field backward stochastic differential equations: a limit approach. The Annals of Probability, 37(4), 1524-1565.

\bibitem{cdl13} Carmona, R., Delarue, F., and Lachapelle, A. (2013). Control of McKean–Vlasov dynamics versus mean field games. Mathematics and Financial Economics, 7(2), 131-166.

\bibitem{Delarue} Delarue, F. (2002). On the existence and uniqueness of solutions to
FBSDEs in a non-degenerate case. Stochastic Processes and their Applications, 99, 209-286.

\bibitem{DM82} Dellacherie, C., and Meyer, P. A. (1982). Probabilities and potential B, Chapter V to VIII. Hermann.

\bibitem{deh19} Djehiche, B., Elie, R., and Hamad{\`e}ne, S. (2019). Mean-field reflected backward stochastic differential equations. arXiv preprint arXiv:1911.06079.

\bibitem{dl1} Dumitrescu, R., Labart C. (2016). Reflected scheme for doubly reflected BSDEs with jumps and RCLL obstacles. Journal of Computational and Applied Mathematics, 296, 827-839.

\bibitem{dl2} Dumitrescu, R., Labart C. (2016). Numerical approximation of doubly reflected BSDEs with jumps and RCLL obstacles. Journal of Mathematical Analysis and Applications 442 (1), 206-243.

\bibitem{dqs15} Dumitrescu, R., Quenez, M. C., and Sulem, A. (2015). Optimal stopping for dynamic risk measures with jumps and obstacle problems. Journal of Optimization Theory and Applications 167 (1), 219-242.

\bibitem{dqs16} Dumitrescu, R., Quenez, M. C., and Sulem, A. (2016). Generalized Dynkin games and doubly reflected BSDEs with jumps. Electronic Journal of Probability, 21, 1-32.


\bibitem{ekppq97} El Karoui, N., Kapoudjian, C., Pardoux, E., Peng, S., and Quenez, M. C. (1997). Reflected solutions of backward SDE's, and related obstacle problems for PDE's. The Annals of Probability, 25(2), 702-737.


\bibitem{essaky08} Essaky, E. H. (2008). Reflected backward stochastic differential equation with jumps and RCLL obstacle. Bulletin des sciences mathematiques, 132(8), 690-710.

\bibitem{g88} G{\"a}rtner, J. (1988). On the McKean-Vlasov limit for interacting diffusions. Mathematische Nachrichten, 137(1), 197-248.


\bibitem{ho03} Hamad{\`e}ne, S., and Ouknine, Y. (2003). Reflected backward stochastic differential equation with jumps and random obstacle. Electronic Journal of Probability, 8.

\bibitem{ho16} Hamad{\`e}ne, S., and Ouknine, Y. (2016). Reflected backward SDEs with general jumps. Theory of Probability \& Its Applications, 60(2), 263-280.

\bibitem{hry19} Hu, K., Ren, Z., and Yang, J. (2019). Principal-agent problem with multiple principals. arXiv preprint arXiv:1904.01413.


\bibitem{j16} Jabin, P. E., and Wang, Z. (2016). Mean field limit and propagation of chaos for Vlasov systems with bounded forces. Journal of Functional Analysis, 271(12), 3588-3627.

\bibitem{jmw} Jourdain, B., M{\`e}l{\`e}ard, S and Woyczynski, W A. (2008). Nonlinear SDEs driven by L{\`e}vy processes and related PDEs, Alea 4, 1–29.

\bibitem{KS98} Karatzas, I. and Shreve, S. E. (1998). Methods of mathematical finance (Vol. 39, pp. xvi+-407). New York: Springer.

\bibitem{kp16} Kruse, T., and Popier, A. (2016). BSDEs with monotone generator driven by Brownian and Poisson noises in a general filtration. Stochastics, 88(4), 491-539.

\bibitem{l18} Lacker, D. (2018). On a strong form of propagation of chaos for McKean-Vlasov equations. Electronic Communications in Probability, 23, 1-11.

\bibitem{lt19} Lauri$\grave{e}$re, M., and Tangpi, L. (2019). Backward propagation of chaos. arXiv preprint arXiv: 1911.06835.

\bibitem{ll07} Lasry, J. M., and Lions, P. L. (2007). Mean field games. Japanese journal of mathematics, 2(1), 229-260.

\bibitem{l14} Li, J. (2014). Reflected mean-field backward stochastic differential equations. Approximation and associated nonlinear PDEs. Journal of Mathematical Analysis and Applications, 413(1), 47-68.

\bibitem{k56} Kac, M. (1956). Foundations of Kinetic Theory. In Third Berkeley Symposium on Mathematical Statistics and Probability (pp. 171-197).

\bibitem{MM14} Marinelli, C., and R{\"o}ckner, M. (2014). On maximal inequalities for purely discontinuous martingales in infinite dimensions. S{\'e}minaire de Probabilit{\'e}s XLVI. Springer, Cham, 2014. 293-315.


\bibitem{m67} McKean, H. P. (1967). Propagation of chaos for a class of non-linear parabolic equations. Stochastic Differential Equations (Lecture Series in Differential Equations, Session 7, Catholic Univ., 1967), 41-57.

\bibitem{pp90} Pardoux, E., and Peng, S. (1990). Adapted solution of a backward stochastic differential equation. Systems \& Control Letters, 14(1), 55-61.

\bibitem{Peng99} Peng, E. (1999). Monotonic limit theorem of BSDE and nonlinear decomposition theorem of Doob–Meyer’s type. Probab. Theory Relat. Fields. 113, 473-499

\bibitem{Peng04} Peng, S. (2004). Nonlinear expectations, nonlinear evaluations and risk measures. In Stochastic methods in finance (pp. 165-253). Springer, Berlin, Heidelberg.

\bibitem{Protter05} Protter, P. E. (2005). Stochastic integration and differential equations. Springer, Berlin, Heidelberg.

\bibitem{Rachev} Rachev, S. T., Klebanov, L., Stoyanov, S. V., and  Fabozzi, F. (2013). The methods of distances in the theory of probability and statistics. Springer Science \& Business Media.

\bibitem{s12} Shkolnikov, M. (2012). Large systems of diffusions interacting through their ranks. Stochastic Processes and their Applications, 122(4), 1730-1747.

\bibitem{s91} Sznitman, A. S. (1991). Topics in propagation of chaos. In Ecole d'{\'e}t{\'e} de probabilit{\'e}s de Saint-Flour XIX—1989 (pp. 165-251). Springer, Berlin, Heidelberg.

\bibitem{tl94} Tang, S., and Li, X. (1994). Necessary conditions for optimal control of stochastic systems with random jumps. SIAM Journal on Control and Optimization, 32(5), 1447-1475.

\bibitem{yw71} Yamada, T., and Watanabe, S. (1971). On the uniqueness of solutions of stochastic differential equations. Journal of Mathematics of Kyoto University, 11(1), 155-167.

\bibitem{Villani} Villani, C (2009). Optimal transport: old and new. Vol. 338. Berlin: Springer.

\end{thebibliography}
\end{document}